\documentclass[twoside,11pt]{article}

\usepackage{stys/jmlr2earxiv}
\usepackage{stys/macro}
\usepackage{mathtools}
\usepackage{graphicx}
\usepackage{subfigure}
\usepackage{booktabs}

\newcommand{\proj}{\mathrm{proj}}
\newcommand{\prox}{\mathrm{prox}}

\jmlrheading{1}{2017}{1--1}{11/17}{??/18}{}{}

\ShortHeadings{Trimmed $\ell_1$ Penalty}{}
\firstpageno{1}

\begin{document}
\title{M-estimation with the Trimmed $\ell_1$ Penalty}
\author{
\name
	Jihun Yun \email arcprime@kaist.ac.kr\\
	\addr School of Computing \\
        Korea Advanced Institute of Science and Technology\\
        Daejeon 34141, Republic of Korea
	\AND
	\name Peng Zheng \email zhengp@uw.edu \\
       \addr Department of Applied Mathematics\\
       University of Washington\\
       Seattle, WA 98195-3925, USA
       \AND
       \name Eunho Yang \email eunhoy@kaist.ac.kr \\
       \addr School of Computing \\
       Korea Advanced Institute of Science and Technology\\
       Daejeon 34141, Republic of Korea
       \AND
       \name Aurelie Lozano \email aclozano@us.ibm.com \\
       \addr IBM T.J. Watson Research Center\\
       Yorktown Heights, NY USA
       \AND
       \name Aleksandr Aravkin \email saravkin@uw.edu \\
       \addr Department of Applied Mathematics\\
       University of Washington\\
       Seattle, WA 98195-3925, USA}
\editor{}
\maketitle
%%%%%%% ABSTRACT %%%%%%%
\begin{abstract}
We study high-dimensional estimators with the trimmed $\ell_1$ penalty,
which leaves the $h$ largest parameter entries penalty-free. 
While optimization techniques for this nonconvex penalty have been studied, the statistical properties have not yet been analyzed.
We present the first statistical analyses for $M$-estimation,
and characterize support recovery, $\ell_\infty$ and $\ell_2$ error of the trimmed $\ell_1$ estimates as a function of the trimming parameter $h$.
Our results show different regimes based on how $h$ compares to the true support size.
Our second contribution is a new algorithm for the trimmed regularization problem, 
which has the same theoretical convergence rate as difference of convex (DC) algorithms, 
but in practice is faster and finds lower objective values. Empirical evaluation of $\ell_1$ trimming for sparse linear regression and graphical model estimation indicate that trimmed $\ell_1$ can outperform vanilla $\ell_1$ and non-convex alternatives.
Our last contribution is to show that the trimmed penalty is beneficial beyond $M$-estimation, and yields promising results for two deep learning tasks: input structures recovery and network sparsification.
\end{abstract}

%%%%%%% INTRODUNCTION %%%%%%%
\section{Introduction}

We consider high-dimensional estimation problems, where the number of variables $p$ can be much larger that the number of observations $n$. 
In this regime, consistent estimation can be achieved by imposing low-dimensional structural constraints on the estimation parameters. 
{\it Sparsity} is a prototypical structural constraint, where at most a small set of parameters can be non-zero.
A key class of sparsity-constrained estimators is based on regularized $M$-estimators using \emph{convex} penalties, with the $\ell_1$ penalty by far the most common. In the context of linear regression, the Lasso estimator~\citep{Tibshirani96} solves an $\ell_1$ regularized or constrained least squares problem, and has strong statistical guarantees, including prediction error consistency~\citep{GeerBuhl09}, consistency of the parameter estimates in some norm~\citep{GeerBuhl09,MeinshausenYu09,Dantzig2007}, and variable selection consistency~\citep{Meinshausen06,Wainwright06,Zhao06}.  
In the context of sparse Gaussian graphical model (GMRF) estimation, the graphical Lasso estimator minimizes the Gaussian negative log-likelihood regularized by the $\ell_1$ norm of the off-diagonal entries of the concentration~\citep{YuaLin07,FriedHasTib2007,BanGhaAsp08}. Strong statistical guarantees for this estimator have been established (see \citet{RWRY11} and references therein).

Recently, there has been significant interest in \emph{non-convex} penalties to alleviate the bias incurred by convex approaches, including SCAD and MCP penalties \citep{FanLi91,breheny2011coordinate,zhang2010nearly,zhang2012general}.  In particular,~\citet{zhang2012general} established consistency for the global optima of least-squares problems with certain non-convex penalties. \citet{LW15} showed that under some regularity conditions on the penalty, any stationary point of the objective function will lie within statistical precision of the underlying parameter vector and thus provide $\ell_2$- and $\ell_1$- error bounds for any stationary point. \citet{LW17} proved that  for a class of \emph{amenable} non-convex regularizers with vanishing derivative away from the origin (including SCAD and MCP), any stationary point is able to recover the parameter support without requiring the typical incoherence conditions needed for convex penalties. All of these analyses apply to 
non-convex penalties that are \emph{coordinate-wise separable}. 

Our starting point is a family of $M$-estimators with trimmed $\ell_1$ regularization, which leaves the largest $h$ parameters unpenalized.
This non-convex family includes the  Trimmed Lasso~\cite{gotoh2017dc,bertsimas2017trimmed} as a special case.  
Unlike SCAD and MCP, trimmed regularization exactly solves constrained best subset selection for large enough values of the regularization parameter, 
and offers more direct control of sparsity via the parameter $h.$ While Trimmed Lasso has been studied from an optimization perspective and with respect to its connections to existing penalties, it has \emph{not} been analyzed from a statistical standpoint.

\vskip 4pt \noindent
{\bf Contributions:}
\begin{itemize}

\item We present the \emph{first} statistical analysis of $M$-estimators with trimmed regularization, \emph{including} Trimmed Lasso. Existing results for non-convex regularizers \citep{LW15,LW17} cannot be applied as trimmed regularization is neither coordinate-wise decomposable nor ``ameanable''.  We provide support recovery guarantees, $\ell_\infty$ and $\ell_2$ estimation error bounds for general $M$-estimators, and derive specialized corollaries for linear regression and graphical model estimation. Our results show different regimes based on how the trimming parameter $h$ compares to the true support size.

\item To optimize the trimmed regularized problem we develop and analyze a new algorithm, which performs better than difference of convex (DC)
functions optimization~\citep{khamaru2018convergence}.  

\item Experiments on sparse linear regression and graphical model estimation show $\ell_1$ trimming is 
 competitive with  other non-convex penalties and vanilla $\ell_1$  when $h$ is selected by cross-validation, and 
  has consistent benefits for a wide range of values for $h$. 

\item Moving beyond $M$-estimation, we apply trimmed regularization to two deep learning tasks: (i) recovering input structures of deep models and (ii) network pruning (a.k.a. sparsification, compression). Our experiments on input structure recovery are motivated by~\citet{Oymak18}, who quantify 
complexity of sparsity encouraging regularizers by introducing the covering dimension, and demonstrates the benefits of regularization for learning over-parameterized networks. 
We show trimmed regularization achieves superior sparsity pattern recovery compared to competing approaches. 
For network pruning, we illustrate the benefits of trimmed $\ell_1$ over vanilla $\ell_1$ on MNIST classification using the LeNet-300-100 architecture. 
 Next, motivated by recently developed pruning methods based on variational Bayesian approaches \citep{dai2018vib,Louizos18}, we propose Bayesian neural networks with trimmed $\ell_1$ regularization. In our experiments, these achieve superior results compared to competing approaches with respect to both error and sparsity level.
  Our work therefore indicates broad relevance of trimmed regularization in multiple problem classes.
\end{itemize}

%%%%%%%% SETUP %%%%%%%
\section{Trimmed Regularization}\label{Sec:Setup}

Trimming has been typically applied to the \emph{loss} function $\L$ of $M$-estimators. 
We can handle outliers
by trimming {\it observations} with large residuals in terms of $\L$: given a collection of $n$ samples, $\Data = \{Z_1, \hdots, Z_n\}$, we solve  
\[\minimize_{\th \in \Omega, \w \in \{0,1\}^n} \sum_{i=1}^n w_i \L(\th; Z_i) \quad \mbox{s.t.} 
\sum_{i=1}^n w_i = n-h,
\] 
where $\Omega$ denotes the parameter space (e.g., $\mathbb{R}^p$ for linear regression). This amounts to trimming $h$ outliers as we learn $\th$ (see~\citet{YLA2018} and references therein).

In contrast, we consider here a family of $M$-estimators with trimmed \emph{regularization} for general high-dimensional problems. We trim entries of $\th$ that incur the largest penalty using the following program:
\begin{align}\label{EqnTrimmedReg1}
	\minimize_{\th \in \Omega, \, \w \in [0,1]^p} \ \ &  \L(\th;\Data) + \lam \sum_{j=1}^p w_j |\theta_j| \nonumber\\
	\st \ \ & {\bf 1}^\top \w \geq p-h \, . 
\end{align}
Defining the order statistics of the parameter $|\theta_{(1)}| > |\theta_{(2)}| > \hdots > |\theta_{(p)}|$, we can partially minimize over $\w$ (setting $w_i$ to $0$ or $1$ based on the size of $|\theta_{i}|$), 
and rewrite the reduced version of problem \eqref{EqnTrimmedReg1} in $\th$ alone:
\begin{align}\label{EqnTrimmedReg2}
	\minimize_{\th \in \Omega} \ \ &  \L(\th;\Data) + \lam \R(\th;h) 
\end{align}
where the regularizer $\R(\th; h)$ is the smallest $p-h$ absolute sum of $\th: \sum_{j=h+1}^p |\theta_{(j)}|$. The constrained version of \eqref{EqnTrimmedReg2}
is equivalent to minimizing a loss subject to a sparsity penalty~\citep{gotoh2017dc}:
$
	\minimize_{\th \in \Omega}  \L(\th;\Data) \  \st \   \|\th\|_0 \leq h.
$
For statistical analysis, we focus on the reduced problem~\eqref{EqnTrimmedReg2}. When optimizing, we exploit the 
structure of~\eqref{EqnTrimmedReg1}, treating weights $\w$ as auxiliary optimization variables, and derive a new fast algorithm 
with a custom analysis that does not use DC structure.

We focus on two key examples: sparse linear models and sparse graphical models. We also present empirical results for trimmed regularization of deep learning tasks to show that the 
ideas and methods generalize well to these areas. 

\paragraph{Example 1: Sparse linear models.} In high-dimensional linear regression, we observe $n$ pairs of a real-valued target  $y_i \in \reals$ and its covariates ${\bm x}_i \in \reals^p$ in a linear relationship:
	\begin{align}\label{EqnLinearModel}
		\y = \X \Tth + \e.	
	\end{align}
 	Here, $\y \in \reals^n$, $\X \in \reals^{n\times p}$ and $\e\in \reals^n$ is a vector of  $n$ independent observation errors. The goal is to estimate the $k$-sparse vector $\Tth \in \reals^p$. According to~\eqref{EqnTrimmedReg2}, we use the least squares loss function with trimmed $\ell_1$ regularizer (instead of the standard $\ell_1$ norm in Lasso \cite{Tibshirani96}):   
	\begin{align}\label{EqnLS}
		\minimize_{\th \in \reals^{p}} \frac{1}{n} \big\| \X \th - \y \big\|_2^2 + \lam \R(\th;h)  	.
	\end{align}

\paragraph{Example 2: Sparse graphical models.} 
GGMs form a powerful class of statistical models for representing distributions over a set of variables~\citep{Lauritzen96}, using undirected graphs to encode conditional independence conditions among variables.
In the high-dimensional setting, graph sparsity constraints are particularly pertinent for estimating GGMs. The most widely used estimator, the graphical Lasso minimizes the negative Gaussian log-likelihood regularized by the $\ell_1$ norm of the entries (or the off-diagonal entries) of the precision matrix (see~\citet{YuaLin07,FriedHasTib2007,BanGhaAsp08}). In our framework, we replace $\ell_1$ norm with its trimmed version:
		$\minimize_{\Th \in \mathcal{S}^{p}_{++}} \ \textrm{trace}\big(\Sig \Th \big) -\log\det \big(\Th\big) + \lam \R(\Th_{\textrm{off}};h)$  	
	where $\mathcal{S}^{p}_{++}$ denotes the convex cone of symmetric and strictly positive definite matrices, $\R(\Th_{\textrm{off}};h)$ does the smallest $p(p-1)-h$ absolute sum of off-diagonals. 

\paragraph{Relationship with SLOPE (OWL) penalty.}
Trimmed regularization has an apparent resemblance to the SLOPE (or OWL) penalty~\citep{bogdan2015slope,figueiredo2014sparse}, but the two are in fact distinct and pursue different goals. Indeed, the SLOPE penalty can be written as $\sum_{i=1}^p w_i |\beta_{(i)}|$ for a \emph{fixed} set of weights $w_1 \geq w_2 \geq \cdots \geq w_p \geq 0$ and where $|\beta_{(1)}|>|\beta_{(2)}|>\cdots> |\beta_{(p)}|$ are the sorted entries of $\bm\beta.$  SLOPE is convex and penalizes more those parameter entries with {\it largest amplitude}, while trimmed regularization is generally non-convex, and only penalizes entries with {\it smallest amplitude}; the weights are also optimization variables. While the goal of trimmed regularization is to alleviate bias, SLOPE is akin to a significance test where top ranked entries are subjected to a ``tougher'' threshold, and has been employed for clustering strongly correlated variables~\citep{figueiredo2014sparse}. Finally from a robust optimization standpoint, Trimmed regularization can be viewed as using an optimistic (min-min) model of uncertainty and SLOPE a pessimistic (min-max) counterpart. We refer the interested reader to~\citet{bertsimas2017trimmed} for an in-depth exploration of these connections. 

\paragraph{Relationship with $\ell_0$ regularization.} 
 The $\ell_0$ norm can be written as $\|\bm\theta\|_0 = \sum_{j=1}^{p} z_j$ with reparameterization $\theta_j = z_j \tilde\theta_j$ such that $z_j \in \{0, 1\}$ and $\tilde\theta_j \neq 0$. ~\citet{Louizos18} suggest a smoothed version via continuous relaxation on $\bm{z}$ in a variational inference framework. The variable $\bm{z}$ plays a similar role to $\bm{w}$ in our formulation in that they both learn sparsity patterns. In Section~\ref{Sec:Exp} we consider a Bayesian extension of the trimmed regularization problem where $\bm\theta$ only is be treated as Bayesian, since we can optimize $\bm{w}$ without any approximation, in contrast to previous work which needs to relax the discrete nature of $\bm{z}$.

%%%%%%% THEORY %%%%%%%
\section{Statistical Guarantees of $M$-Estimators with Trimmed Regularization}

Our goal is to estimate the \emph{true} $k$-sparse parameter vector (or matrix) $\Tth$ that is the minimizer of expected loss: $\Tth := \argmin_{\th \in \Omega} \E[\L(\th)]$. We use $\Supp$ to denote the support set of $\Tth$, namely the set of non-zero entries (i.e., $k = |\Supp|$). In this section, we derive support recovery, $\ell_\infty$ and $\ell_2$ guarantees under the following standard assumptions: 

\begin{enumerate}[leftmargin=0.5cm, itemindent=0.65cm,label=\textbf{(C-$\bf{\arabic*}$)}, ref=\textnormal{(C-$\arabic*$)},start=1]
	\item The loss function $\L$ is differentiable and convex.\label{Con:diff}
	\vspace{-.2cm}\item {\bf (Restricted strong convexity on $\th$)} Let $\errtSet$ be the possible set of error vector on the parameter $\th$. Then, for all $\errt := \th - \Tth \in \errtSet$,
		$\Biginner{\nabla \L(\Tth + \errt) - \nabla\L(\Tth) }{\errt} \geq \, \RSCcon \|\errt\|_2^2 - \RSCtolOne \frac{\log p}{n}\|\errt\|_1^2$, 
		where $\RSCcon$ is a ``curvature'' parameter, and $\RSCtolOne$
		is a ``tolerance'' constant.\label{Con:rsc}
\end{enumerate}

In the high-dimensional setting ($p>n$), the loss function $\L$ cannot be strongly convex in general. \ref{Con:rsc} imposes strong curvature only in some limited directions where the ratio $\frac{\|\errt\|_1}{\|\errt\|_2}$ is small. This condition has been extensively studied and known to hold for several popular high dimensional problems (see \citet{Raskutti2010,NRWY12,LW15} for instance). The convexity condition of $\L$ in \ref{Con:diff} can be relaxed as shown in \cite{LW17}. For clarity, however, we focus on convex loss functions. 

We begin with $\ell_\infty$ guarantees. We use a primal-dual witness (PDW) proof technique, which we adapt to the trimmed regularizer $\R(\th;h)$. The PDW method has been used to analyze the support set recovery of $\ell_1$ regularization \citep{Wainwright2006new,YRAL13} as well as decomposable and amenable non-convex regularizers \citep{LW17}. However, the trimmed regularizer $\R(\th;h)$ is neither decomposable nor amenable, thus the results of~\citet{LW17} cannot be applied. The key step of PDW is to build a restricted program: Let $\Nonreg$ be an arbitrary subset of $\{1,\hdots,p\}$ of size $h$. Denoting $\USupNonreg := \Supp \cup \Nonreg$ and $\DSupNonreg := \Supp - \Nonreg$, we consider the following restricted program:  
	$\Gth \in \argmin_{\th \in \reals^{\USupNonreg} : \ \th \in \Omega}  \   \L(\th) + \lam \R(\th;h)$ 
where we fix $\Gth_j = 0$ for all $j \in \USupNonregC$. We further construct the dual variable $\Gz$ to satisfy the zero sub-gradient condition
\begin{align}\label{EqnPDW}
	\nabla \L(\Gth) + \lam \Gz = 0	
\end{align}
where $\Gz = (0, \Gz_{\DSupNonreg}, \Gz_{\USupNonregC})$ for $\Gth = (\Gth_{\Nonreg}, \Gth_{\DSupNonreg},0_{\USupNonregC})$ (after re-ordering indices properly) and $\Gz_{\DSupNonreg} \in \partial \|\Gth_{\DSupNonreg}\|_1$. We suppress the dependency on $\Nonreg$ in $\Gz$ and $\Gth$ for clarity. In order to derive the final statement, we will establish the strict dual feasibility of $\Gz_{\USupNonregC}$, i.e., $\|\Gz_{\USupNonregC}\|_\infty < 1$. 

The following theorem describes our main theoretical result concerning \emph{any} local optimum of the non-convex program \eqref{EqnTrimmedReg2}. The theorem guarantees under strict dual feasibility that non-relevant parameters of local optimum have smaller absolute values than relevant parameters; hence relevant parameters are not penalized (as long as $h \geq k$).    
\begin{theorem}\label{ThmSupp}
	Consider the problem with trimmed regularizer \eqref{EqnTrimmedReg2} that satisfies \ref{Con:diff} and \ref{Con:rsc}. Let $\Lth$ be an any local minimum of \eqref{EqnTrimmedReg2} with a sample size $n \geq \frac{2\RSCtolOne}{\RSCcon} (k+h)\log p$ and $\lam \geq 2 \|\nabla \L(\Tth) \|_\infty$. Suppose that:
	\begin{enumerate}[leftmargin=0.25cm, itemindent=0.45cm,label=(\alph*)]
		\item given {\bf any} selection of $\Nonreg \subseteq \{1,\hdots,p\}$ s.t. $|\Nonreg|=h$, the dual vector $\Gz$ from the PDW construction \eqref{EqnPDW} satisfies the strict dual feasibility with some $\delta \in \left(0, 1\right]$, 
				$\|\Gz_{\USupNonregC}\|_\infty \leq 1 - \delta$
			where $\USupNonreg$ is the union of true support $S$ and $\Nonreg$,
		\vspace{-.2cm}\item letting $\Q : = \int_0^1 \nabla^2 \L\big(\Tth + t(\Gth - \Tth) \big) dt$, the minimum absolute value $\Tth_{\min} := \min_{j\in \Supp} |\Tth_j|$ is lower bounded by

			$\frac{1}{2}\Tth_{\min} \geq  \| (\Qs)^{-1} \nabla\L(\Tth)_\USupNonreg \|_{\infty}   + \lam \matnorm{(\Qs)^{-1}}{\infty}$ where $\matnorm{\cdot}{\infty}$ denotes the maximum absolute row sum of the matrix. 
	\end{enumerate}
  	Then, the following properties hold:
  	 \begin{enumerate}[leftmargin=0.25cm, itemindent=0.45cm,label=(\arabic*)]
 		\item For every pair $j_1 \in \Supp, j_2 \in \SuppC$, we have $|\Lth_{j_1}| > |\Lth_{j_2}|$,
 		\vspace{-.4cm}\item If $h < k$, all $j \in \SuppC$ are successfully estimated as zero and $\|\Lth -\Tth\|_\infty$ is upper bounded by 
 			\begin{align}\label{EqnThmInfty}
 				\big\| \big(\TQ\big)^{-1} \nabla\L(\Tth)_{\Supp} \big\|_{\infty} + \lam \matnormbig{\big(\TQ\big)^{-1}}{\infty} ,
 			\end{align}
 		 \vspace{-.4cm}\item If $h \geq k$, at least the smallest (in absolute value) $p-h$ entries in $\SuppC$ are estimated exactly as zero and we have a simpler (possibly tighter) bound:
 			\begin{align}\label{EqnThmInftyTight}
				\|\Lth -\Tth\|_\infty \leq \big\| \big(\GQs\big)^{-1} \nabla\L(\Tth)_{\GUSupNonreg} \big\|_{\infty}
 			\end{align}
 			where $\GUSupNonreg$ is defined as the $h$ largest absolute entries of $\Lth$ including $\Supp$.
 	\end{enumerate}

\end{theorem}

\paragraph{Remarks.} The above theorem will be instantiated for the specific cases of sparse linear and sparse graphical models in subsequent corollaries (for which we will bound terms involving $\nabla\L(\Tth)$, $\Gz$ and $\Q$). Though
conditions (a) and (b) in Theorem \ref{ThmSupp} seem apparently more stringent than the case where $h=0$ (vanilla Lasso), we will see in corollaries that they are uniformly upper bounded for all selections, under the asymptotically same probability as $h=0$.  

Note also that for $h=0$, we recover the results for the vanilla $\ell_1$ norm. Furthermore, by the statement $(1)$ in the theorem, if $h<k$, $\GUSupNonreg$ only contains relevant feature indices and some relevant features are not penalized. If $h \geq k$, $\GUSupNonreg$ includes all relevant indices (and some non-relevant indices). In this case, the second term in \eqref{EqnThmInfty} disappears, but the term $\big\|\big(\GQs\big)^{-1} \nabla\L(\Tth)_{\GUSupNonreg} \big\|_{\infty}$ increases as $|\GUSupNonreg|$ gets larger. Moreover, the condition that $n \asymp (k+h) \log p$ will be violated as $h$ approaches $p$. While we do not know the true sparsity $k$ a priori in many problems, we implicitly assume that we can set $h \asymp k$ (i.e., by cross-validation). 

Now we turn to $\ell_2$ bound under the same conditions:
\begin{theorem}\label{ThmL2}
	Consider the problem with trimmed regularizer \eqref{EqnTrimmedReg2} where all conditions in Theorem \ref{ThmSupp} hold. Then, for any local minimum of \eqref{EqnTrimmedReg2}, the parameter estimation error in terms of $\ell_2$ norm is upper bounded: for some constant $\Cltwo$,
\[
		\| \Lth - \Tth \|_2  \leq 
	\begin{dcases*}
        \Cltwo \lam \left(\sqrt{k}/2 + \sqrt{k-h}\right)  & if $h < k$\\
        \Cltwo \lam \sqrt{h}/2  & otherwise
        \end{dcases*}
\]
\end{theorem}

\paragraph{Remarks.}
The benefit of using trimmed $\ell_1$ over standard $\ell_1$ can be clearly seen in Theorem \ref{ThmL2}. Even though both have the same asymptotic convergence rates (in fact, standard $\ell_1$ is  already information theoretically optimal in many cases such as high-dimensional least squares), trimmed $\ell_1$ has a smaller constant: $\frac{3\Cltwo \lam\sqrt{k}}{2}$ for standard $\ell_1$ ($h=0$) vs. $\frac{\Cltwo \lam \sqrt{k}}{2}$ for trimmed $\ell_1$ ($h=k$). 
Comparing with non-convex $(\mu,\gamma$)-amenable regularizers SCAD or MCP, we can also observe that the estimation bounds are asymptotically the same: $\| \Lth - \Tth \|_{\infty} \leq c \| (\TQ)^{-1} \nabla\L(\Tth)_{\Supp} \|_{\infty}$ and $\| \Lth - \Tth \|_2 \leq c \lam \sqrt{k}$. However, the constant $c$ here for those regularizers might be too large if $\mu$ is not small enough, since it involves $\frac{1}{\RSCcon - \mu}$ term (vs. $\frac{1}{\RSCcon}$ for the trimmed $\ell_1$.)
Moreover amenable non-convex regularizers require the additional constraint $\|\th\|_1 \leq R$ in their optimization problems for theoretical guarantees, along with further assumptions on $\Tth$ and tuning parameter $R$,
and the true parameter must be feasible for their modified program (see \citet{LW17}). 
The condition $\|\bm \theta^*\|_1 \le R$ is stringent with respect to the analysis: as $p$ and $k$ increase, in order for $R$ to remain constant, $\|\bm \theta^*\|_\infty$ must shrink to get satisfactory theoretical bounds. 
In contrast, while choosing the trimming parameter $h$ requires cross-validation, it is possible to set $h$ on a similar order as $k$.

We are now ready to apply our main theorem to the popular high-dimensional problems introduced in Section \ref{Sec:Setup}: sparse linear regression and sparse graphical model estimation. Due to space constraint, the results for sparse graphical models are provided in the supplementary materials.  

\subsection{Sparse Linear Regression} \label{sec:slr}
Motivated by the information theoretic bound for arbitrary methods, all previous analyses of sparse linear regression assume $n \geq c_0 k \log p$ for sufficiently large constant $c_0$. We also assume $n \geq c_0 \max\{k,h\} \log p$, provided $h \asymp k$.
\begin{corollary}\label{CorLS}
	Consider the model \eqref{EqnLinearModel} where $\e$ is sub-Gaussian. Suppose we solve \eqref{EqnLS} with the selection of:
 	\begin{enumerate}[leftmargin=0.25cm, itemindent=0.45cm,label=(\alph*)]
 		\item $\lam \geq \cLStwo \sqrt{\frac{\log p}{n}}$ for some constant $\cLStwo$ depending only on the sub-Gaussian parameters of $X$ and $\e$
 		\item $h$ satisfying: for any selection of $\Nonreg \subseteq [p] \text{ s.t. } |\Nonreg| = h$,
	 		\begin{align}\label{EqnLSIncoh}
				& \matnormBig{\big(\GG^{-1}\big)_{\USupNonreg \USupNonreg}}{\infty}  \leq \cLSone, \qquad \matnormBig{\GG_{\USupNonregC\USupNonreg} \Big(\GG_{\USupNonreg\USupNonreg}\Big)^{-1}}{\infty} \leq \eta, \nonumber\\
				&\max\Big\{\lambda_{\max}(\GG_{\USupNonregC \USupNonregC}),\lambda_{\max}\big((\GG_{\USupNonreg \USupNonreg})^{-1}\big)\Big\} \leq \cLSsix 				 
			\end{align}
			where $\GG = \frac{\X^\top X}{n}$ is the sample covariance matrix and $\lambda_{\max}$ is the maximum singular value of a matrix.
	\end{enumerate}
 	Further suppose $\frac{1}{2}\Tth_{\min} \geq \cLSthree \sqrt{\frac{\log p}{n}} + \lam \cLSone$ for some constant $\cLSthree$.
 	Then with high probability at least $1-\cLSfour \exp (-\cLSfive \log p)$, any local minimum $\Lth$ of \eqref{EqnLS} satisfies 
 	\begin{enumerate}[leftmargin=0.25cm, itemindent=0.45cm,label=(\alph*)]
 		\item for every pair $j_1 \in \Supp, j_2 \in \SuppC$, we have $|\Lth_{j_1}| > |\Lth_{j_2}|$,
 		\item if $h < k$, all $j \in \SuppC$ are successfully estimated as zero and we have
 			\begin{align*}
 				& \|\Lth -\Tth\|_\infty  \leq \cLSthree \sqrt{\frac{\log p}{n}} + \lam \cLSone , \nonumber\\
 				& \| \Lth - \Tth \|_2  \leq c_4 \sqrt{\frac{\log p}{n}} \left(\sqrt{k}/2 + \sqrt{k-h}\right) \, .  			
 			\end{align*}
 		\item if $h \geq k$, at least the smallest $p-h$ entries in $\SuppC$ have exactly zero and we have
 			\begin{align*}
				\|\Lth -\Tth\|_\infty \leq \cLSthree \sqrt{\frac{\log p}{n}} , \ \ \| \Lth - \Tth \|_2  \leq \frac{c_4}{2} \sqrt{\frac{h \log p }{n}} \, .
 			\end{align*}
  	\end{enumerate}
\end{corollary}

\paragraph{Remarks.} The conditions in Corollary \ref{CorLS} are also used in previous work and may be shown to hold with high probability via standard concentration bounds for sub-Gaussian matrices. In particular \eqref{EqnLSIncoh} is known as an incoherence condition for sparse least square estimators \citep{Wainwright06_new}. In the case of vanilla Lasso, estimation will fail if the incoherence condition is violated \citep{Wainwright06_new}. 
In contrast, we confirm by simulations in Section \ref{Sec:Exp} that the trimmed $\ell_1$ problem \eqref{EqnLS} can succeed even when this condition is not met. Therefore we conjecture that the incoherence condition could be relaxed in our case, similarly to the case of non-convex $\mu$-amenable regularizers such as SCAD or MCP \citep{LW17}. Proving this conjecture is highly non-trivial, since our penalty is based on a sum of absolute values, which is not $\mu$-amenable; we leave the proof for future work.

%%%%%%% OPTIMIZATION %%%%%%%

\begin{algorithm}[t]
\caption{Block Coordinate Descent for \eqref{EqnTrimmedReg1}}
\label{alg:bcd}
\begin{algorithmic}
\State {\bfseries Input:} $\lambda$, $\eta$, and $\tau$.
\State {\bfseries Initialize:} $\bm \theta^0$, $\bm w^0$, and $k=0$.
\While{not converged}
\Let{$\bm w^{k+1}$}{$\mathrm{proj}_{\cS}[\bm w^k - \tau \bm r(\bm \theta^{k})]$}
\Let{$\bm \theta^{k+1}$}{$\mathrm{prox}_{\eta\lambda\R(\cdot,\bm w^{k+1})}[\bm \theta^k - \eta \nabla \L(\bm\theta^k)]$}
\Let{$k$}{$k+1$}
\EndWhile
\State {\bfseries Output:} $\bm \theta^k$, $\bm w^k$.
\end{algorithmic}
\end{algorithm}

We develop and analyze a block coordinate descent algorithm for solving objective \eqref{EqnTrimmedReg1}, which is highly nonconvex problem because of the coupling of $w$ and $\theta$ in the regularizer. 
The block-coordinate descent algorithm uses simple nonlinear operators: 
\[
\begin{aligned}
\proj_{\mathcal{S}}(\bm z) &:= \arg\min_{\bm w \in \mathcal{S}} \frac{1}{2}\|\bm z - \bm w\|^2\\
\prox_{\eta\lambda \mathcal{R}(\cdot, \bm w^{k+1})} (\bm z) &:= \arg\min_{\bm \theta} \frac{1}{2\eta\lambda} \|\bm \theta - \bm z \|^2 + \sum_{j=1}^p w_j^{k+1} |\theta_{j}|
\end{aligned}
\]
Adding a block of weights $\bm w$ decouples  the problem into simply computable pieces.
Projection onto a polyhedral set is straightforward, while the prox operator is a weighted soft thresholding step. 

We analyze Algorithm~\ref{EqnTrimmedReg1} using the structure of~\eqref{EqnTrimmedReg1} instead of relying on the DC formulation for~\eqref{EqnTrimmedReg2}. 
The convergence analysis is summarized in Theorem~\ref{th:con} below. The analysis centers on the  
general objective function
\begin{equation}
\label{eq:model}
\min_{\bm \theta, \bm w} F(\bm \theta, \bm w) := \L(\bm \theta) + \lambda \sum_{i=1}^p w_i r_i(\bm \theta) + \delta(\bm w| \cS), 
\end{equation}
where  $\delta(\bm w| \cS)$ enforces $w \in \cS$. We let 
\[
\bm r(\bm \theta) = \begin{bmatrix}
r_1(\bm x)&
\dots&
r_p(\bm x)
\end{bmatrix}^T, \R(\bm \theta, \bm w) = \inner{\bm w}{\bm r(\bm \theta)}.
\]
\textcolor{black}{In the case of trimmed $\ell_1,$ $r$ is the $\ell_1$ norm, $r_i(x)=|x_i|$ and $\cS$ encodes 
the constraints $0 \leq w_i \leq 1$, $\bm 1^T\bm w= p-h$.}

We make the following assumptions.
\begin{assumption}
\label{assumption}
(a)  $\L$ is a smooth closed convex function with an $L_f$-Lipchitz continuous gradient; (b) $r_i$ are convex, and $L_r$-Lipchitz continuous and
(c) $\cS$ is a closed convex set and $F$ is bounded below.
\end{assumption}

In the non-convex setting, we do not have access to distances to optimal iterates or best function values, 
as we do for strongly convex and convex problems. Instead, we use distance to stationarity to analyze the algorithm.
Objective~\eqref{eq:model} is highly non-convex, so we design a stationarity criterion, which 
goes to $0$ as we approach stationary points.  The analysis then shows 
Algorithm~\ref{EqnTrimmedReg1} drives this measure to $0$, i.e. converges to stationarity. 
In our setting, every stationary point of~\eqref{EqnTrimmedReg1}
corresponds to a local optimum in $\bm w$ with $\bm \theta$ fixed, 
and a local optimum in $\bm \theta$ with $\bm w$ fixed. 

\begin{definition}[Stationarity]
Define the stationarity condition $T(\bm \theta, \bm w)$ by
\begin{equation}
\label{eq:stationarity}
\begin{aligned}
T(\bm \theta, \bm w) = \min\{\|\bm u\|^2 + \|\bm v\|^2 : &\bm u \in \partial_\theta F(\bm \theta, \bm w),\\
&\bm v \in \partial_w F(\bm \theta, \bm w)\}.
\end{aligned}
\end{equation}
The pair $(\bm \theta, \bm w)$ is a stationary point when $T(\bm \theta, \bm w) = 0$.
\end{definition}

\begin{theorem}
\label{th:con}
Suppose Assumptions~\ref{assumption} (a-c) hold, and define the quantity $\mathcal{G}$ as follows:
\[
\mathcal{G}_k := \frac{L_f}{2}\|\bm \theta^{k+1} - \bm \theta^k\|^2 + \frac{\lambda}{\tau}\|\bm w^{k+1}- \bm w^k\|^2.
\] 
With step size $\eta = 1/L_f$, we have,
\[
\begin{aligned}
\min_{k} \mathcal{G}_k &\leq \frac{1}{K}\sum_{k=1}^K\mathcal{G}_k \le \frac{1}{K}(F(\bm \theta^1) - F^*) \\
T(\bm \theta^{k+1}, \bm w^{k+1}) &\le (4 + 2\lambda L_r/L_f) \mathcal{G}_k,
\end{aligned}
\]
and therefore
\[
\min_{k = 1: K} \{T(\bm \theta^{k}, \bm w^{k})\} \leq\frac{4 + 2\lambda L_r/L_f}{K}(F(\bm \theta^1) - F^*).
\]
\end{theorem}

\textcolor{black}{The trimmed $\ell_1$ problem satisfies Assumption~\ref{assumption} and hence Theorem~\ref{th:con} holds.}
Algorithm~\ref{alg:bcd} for~\eqref{EqnTrimmedReg1} 
converges at a sublinear rate measured using the distance to stationarity $T$~\eqref{eq:stationarity}, see Theorem~\ref{th:con}. 
In the simulation experiments of Section~\ref{Sec:Exp}, we will observe that the iterates converge to very close points regardless of initializations.
\citet{khamaru2018convergence} use similar concepts to analyze their DC-based algorithm, since it is also developed 
for a nonconvex model.
\iffalse In the supplements, we include a small numerical experiment, comparing Algorithm 1 
with~Algorithm 2 of~\cite{khamaru2018convergence} (this prox-type algorithm did particularly well for subset selection, see Figure 2 of~\cite{khamaru2018convergence}). Initial progress of the methods is comparable, 
but Algorithm~\eqref{alg:bcd} continues at a linear rate to a lower value of the objective, while 
Algorithm 2 of~\cite{khamaru2018convergence} tapers off at a higher objective value. 
We consistently observe this phenomenon for a broad range of settings, regardless of hyperparameters.
\fi 

We include a small numerical experiment, comparing Algorithm 1 
with~Algorithm 2 of~\cite{khamaru2018convergence}. 
The authors proposed multiple approaches for DC programs; the prox-type algorithm (Algorithm 2) did particularly well for subset selection, see Figure 2 of~\cite{khamaru2018convergence}.
We generate Lasso simulation data with variables of dimension $500$, and $100$ samples. 
The number of nonzero elements in the true generating variable is 10. We take $h=25$,
and apply both Algorithm~\ref{alg:bcd} and Algorithm 2 of \cite{khamaru2018convergence}. 
%Result are shown in Figure~\ref{fig:alg_compare}. 
Initial progress of the methods is comparable, 
but Algorithm~\ref{alg:bcd} continues at a linear rate to a lower value of the objective, while 
Algorithm 2 of~\cite{khamaru2018convergence} tapers off at a higher objective value. 
We consistently observe this phenomenon for a broad range of settings, regardless of hyperparameters; 
see convergence comparisons in Figure~\ref{fig:alg_compare} for  $\lambda \in \{0.5, 5, 20\}$. 
This comparison is very brief; we leave a detailed study comparing Algorithm~\ref{alg:bcd} with 
DC-based algorithms to future algorithmic work, along with further analysis of Algorithm~\ref{alg:bcd} and its variants 
under the Kurdyka-Lojasiewicz assumption~\citep{attouch2013convergence}.

%%%%%%% EXPERIMENTS %%%%%%%
\section{Experimental Results}\label{Sec:Exp}

\paragraph{Simulations for sparse linear regression.}

\begin{figure}[t]
\vskip 0pt
\centering
\subfigure[$\lambda = 0.5$]{\includegraphics[width=0.32\linewidth]{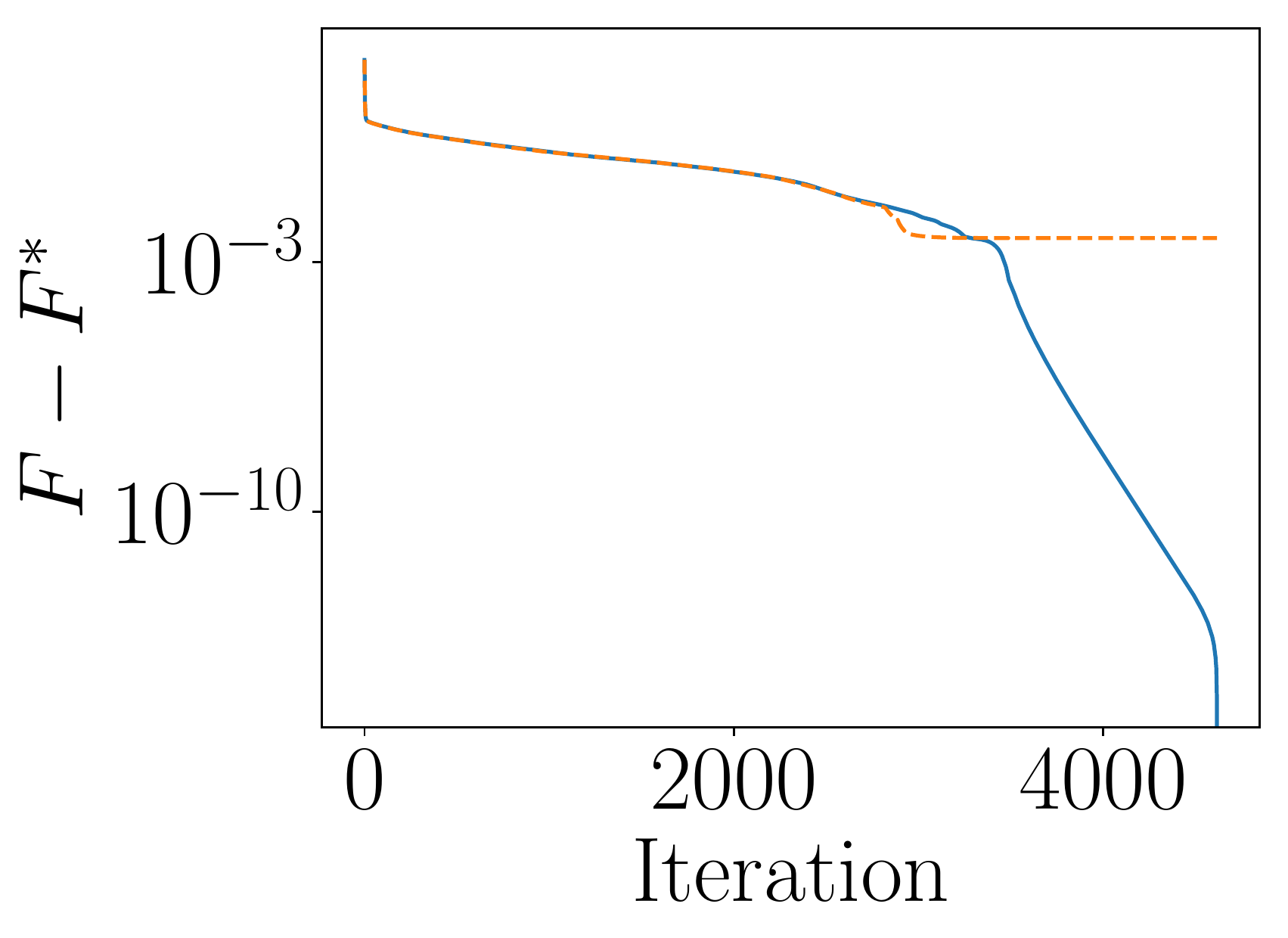}}
\subfigure[$\lambda = 5$]{\includegraphics[width=0.32\linewidth]{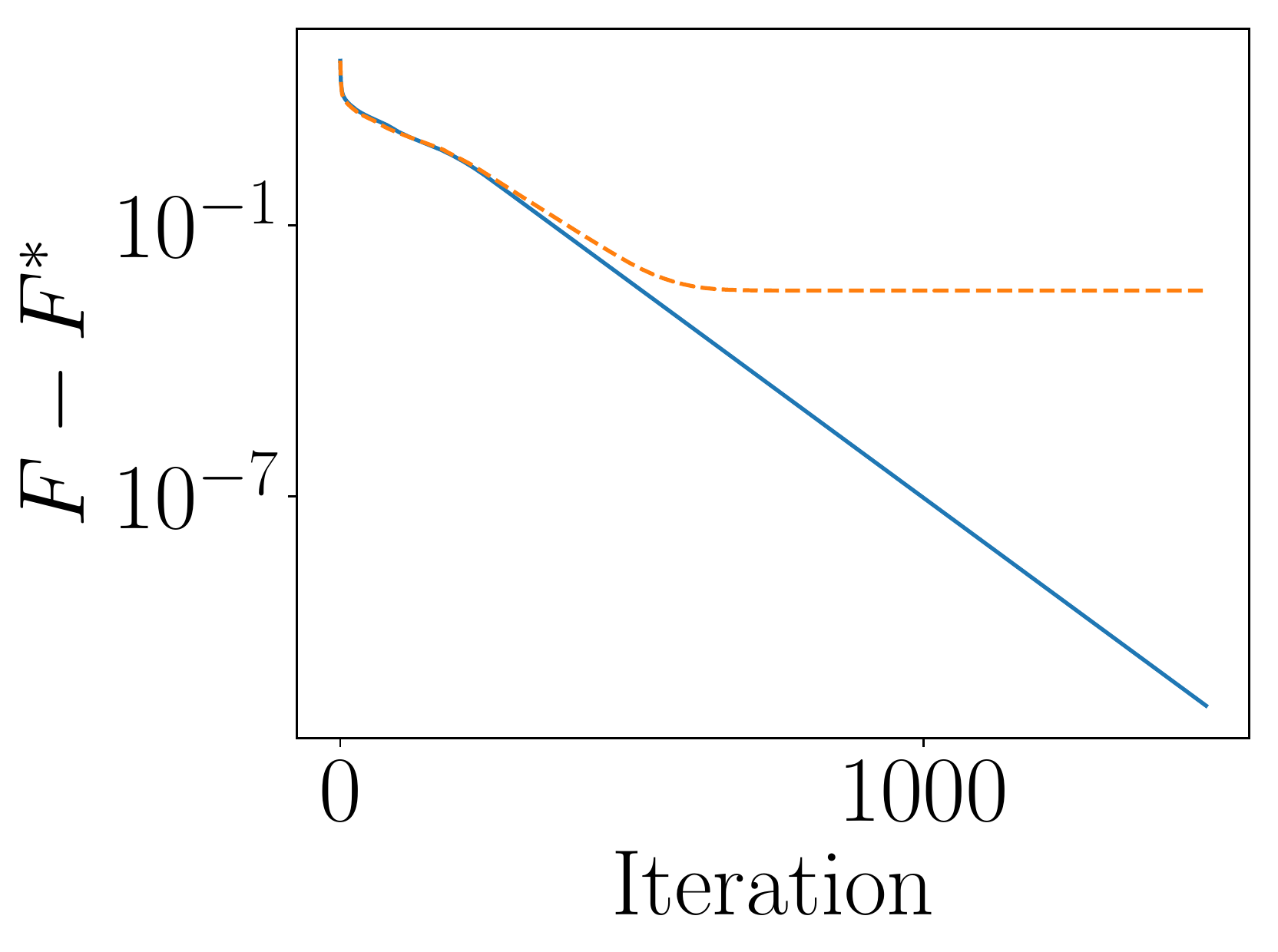}}
\subfigure[$\lambda = 20$]{\includegraphics[width=0.32\linewidth]{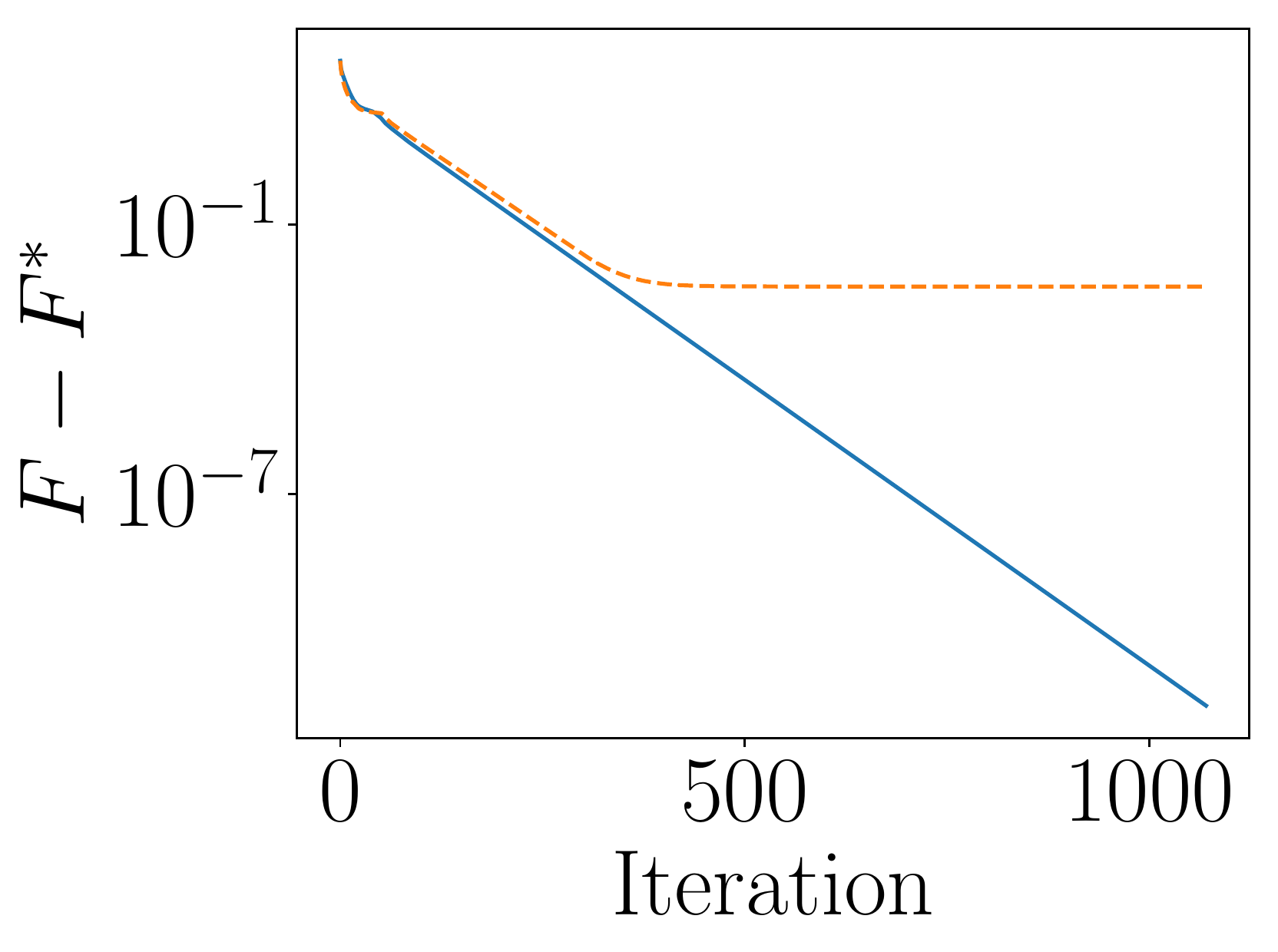}}
\caption{Convergence of Algorithm~\ref{alg:bcd} (blue solid) vs. Algorithm 2 of \cite{khamaru2018convergence} (orange dot). We see consistent results across parameter settings.}
\label{fig:alg_compare}
\end{figure}

\begin{figure}[t]
	   \setlength{\belowcaptionskip}{10pt}
		\centering
		\subfigure[$p=128, k=8$]{\includegraphics[width=0.32\linewidth]{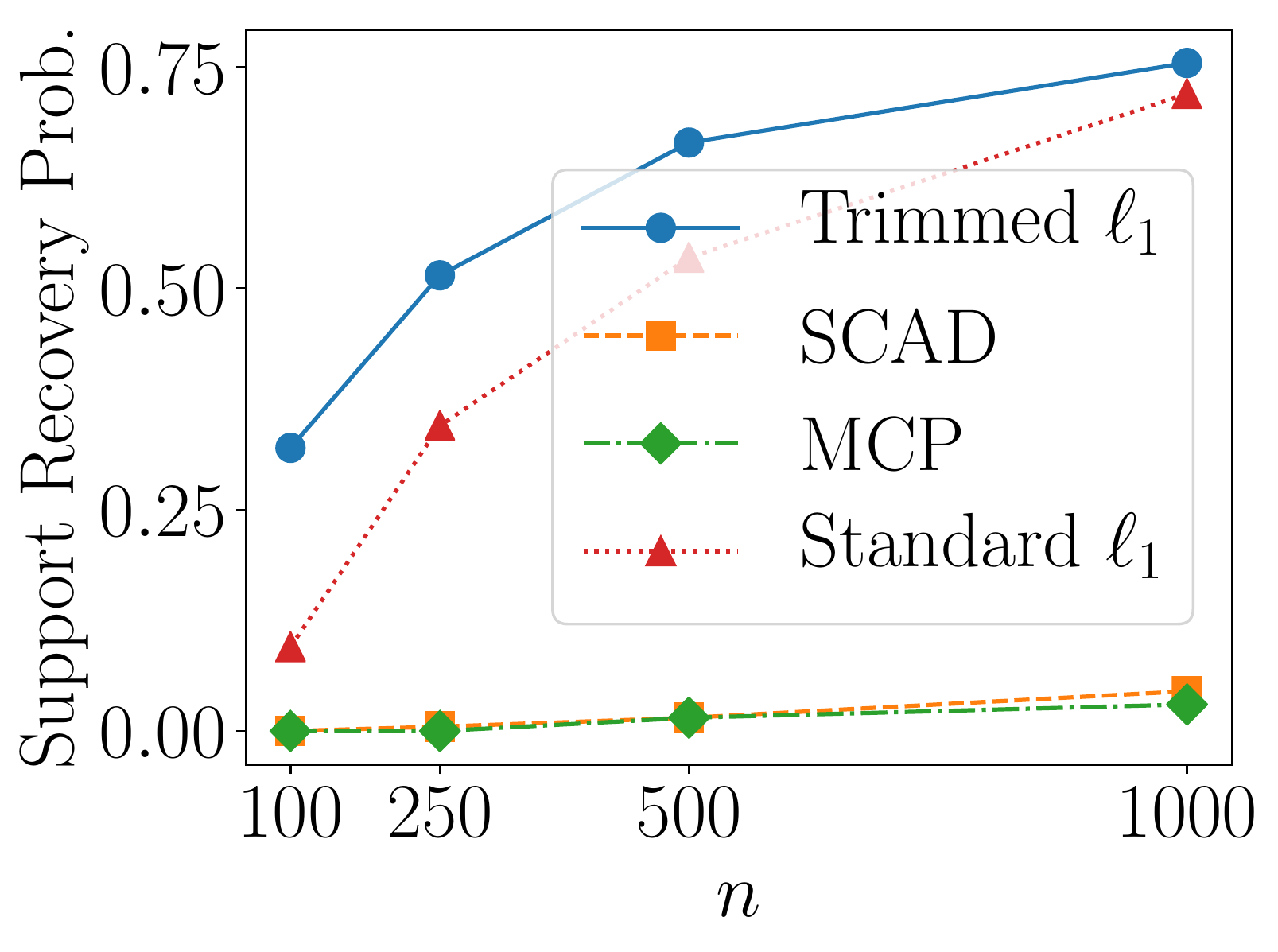}}
		\subfigure[$p=256, k=16$]{\includegraphics[width=0.32\linewidth]{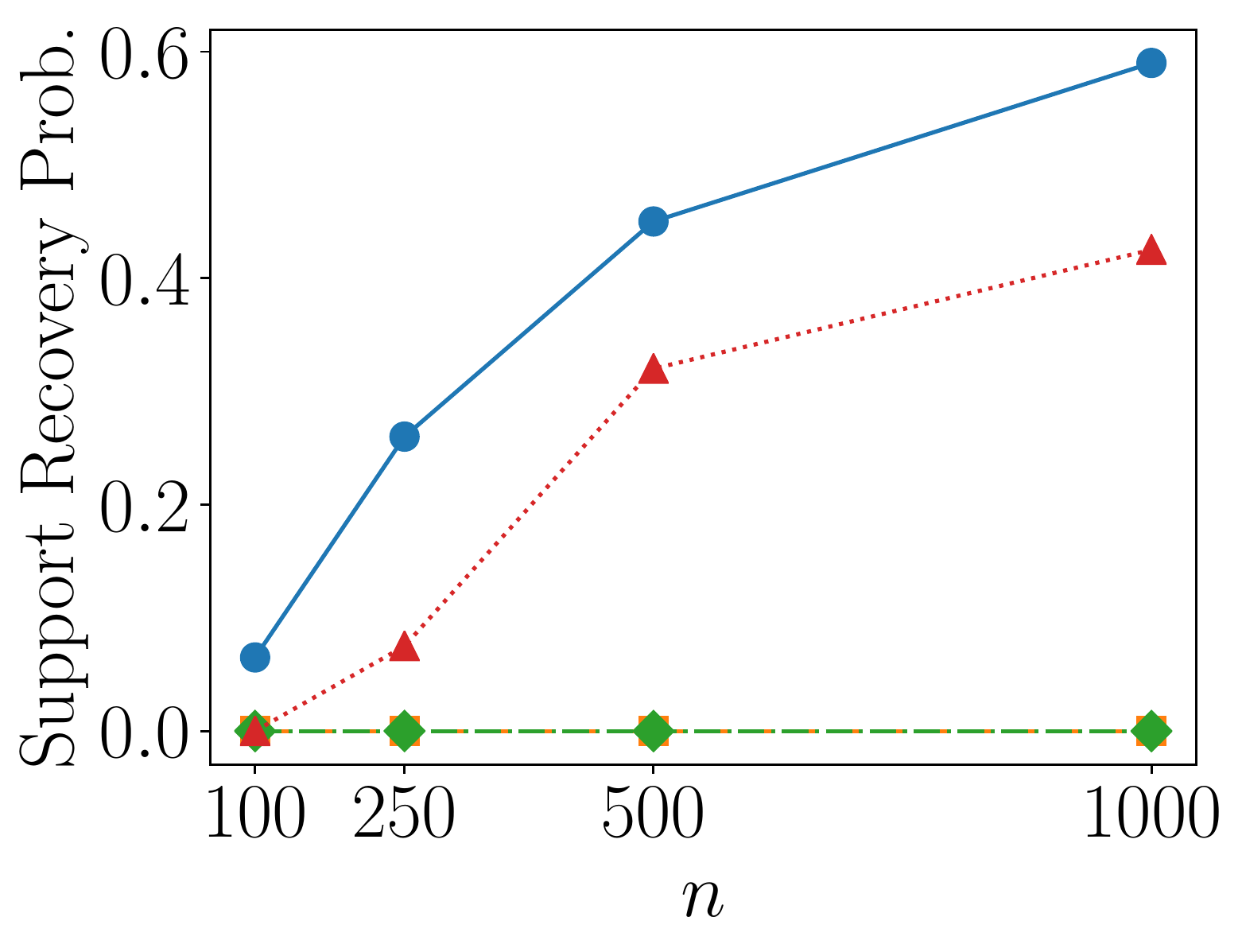}}
		\subfigure[$p=512, k=32$]{\includegraphics[width=0.32\linewidth]{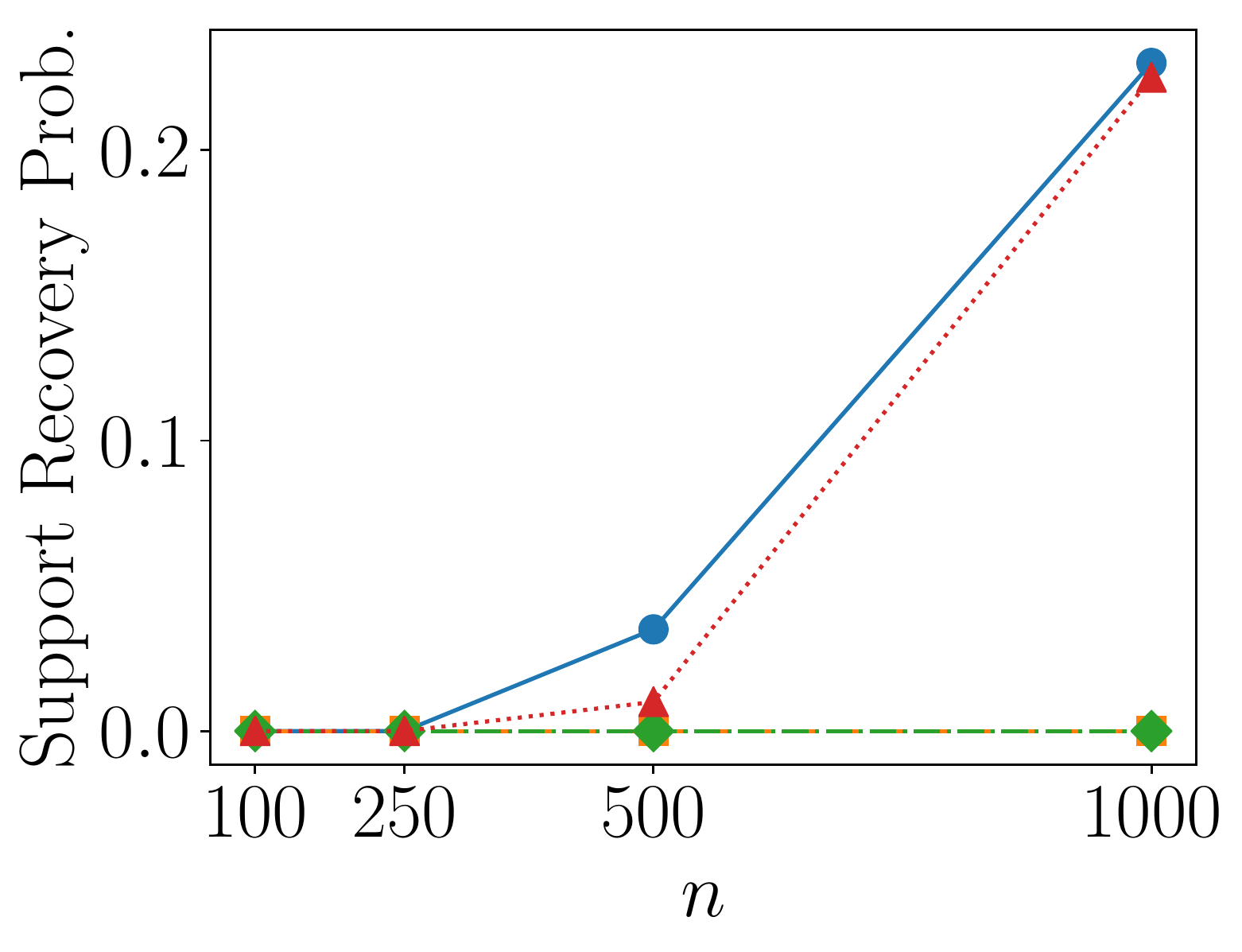}}
		\subfigure[Stationarity]{\includegraphics[width=0.32\linewidth]{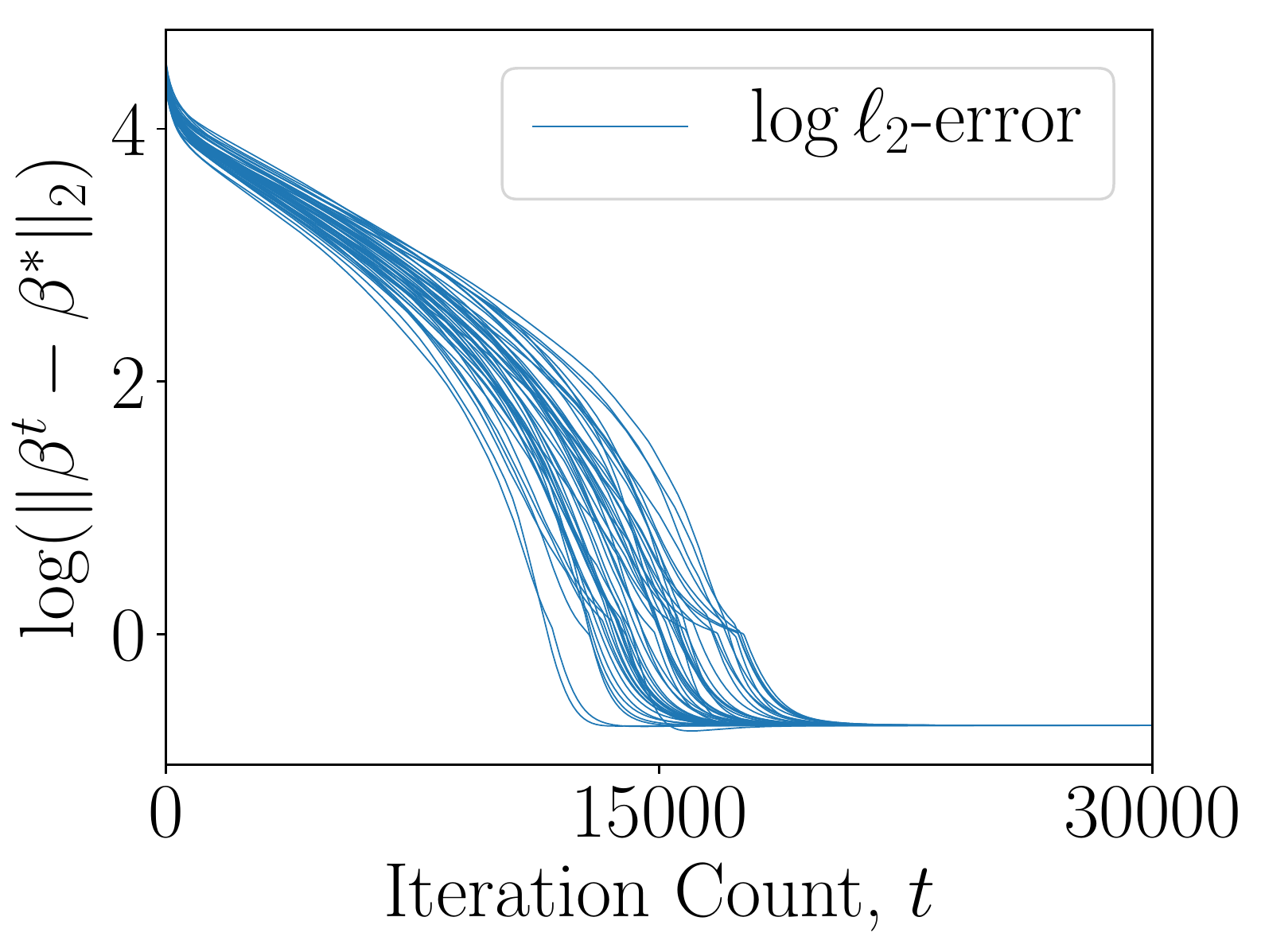}} 
		\subfigure[$\log \ell_2$-errors]{\includegraphics[width=0.32\linewidth]{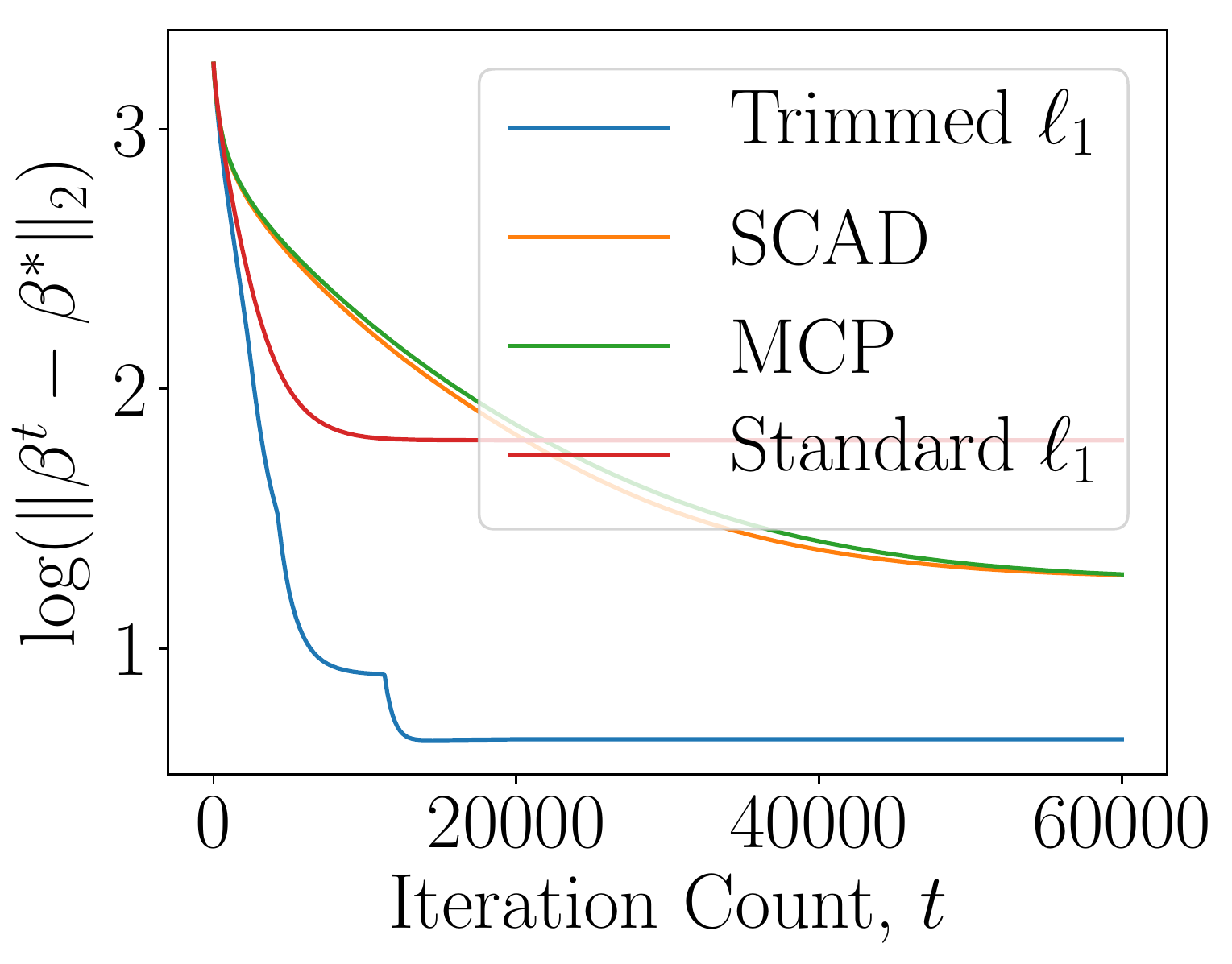}}
		\caption{Results for the incoherent case of the first experiments. \textbf{(a)}$\sim$\textbf{(c)}: Probability of sucessful support recovery for Trimmed $\ell_1$, SCAD, MCP, and standard $\ell_1$ as sample size $n$ increases. For \textbf{(d)}, \textbf{(e)}, we adopt the high-dimensional setting with $(n, p, k) = (160, 256, 16)$, and use 50 random initializations.}
		\label{Fig:incoherence_reg}
		
\end{figure}
	
\begin{figure}[t]
		\centering
		\subfigure[$p=128, k=8$]{\includegraphics[width=0.32\linewidth]{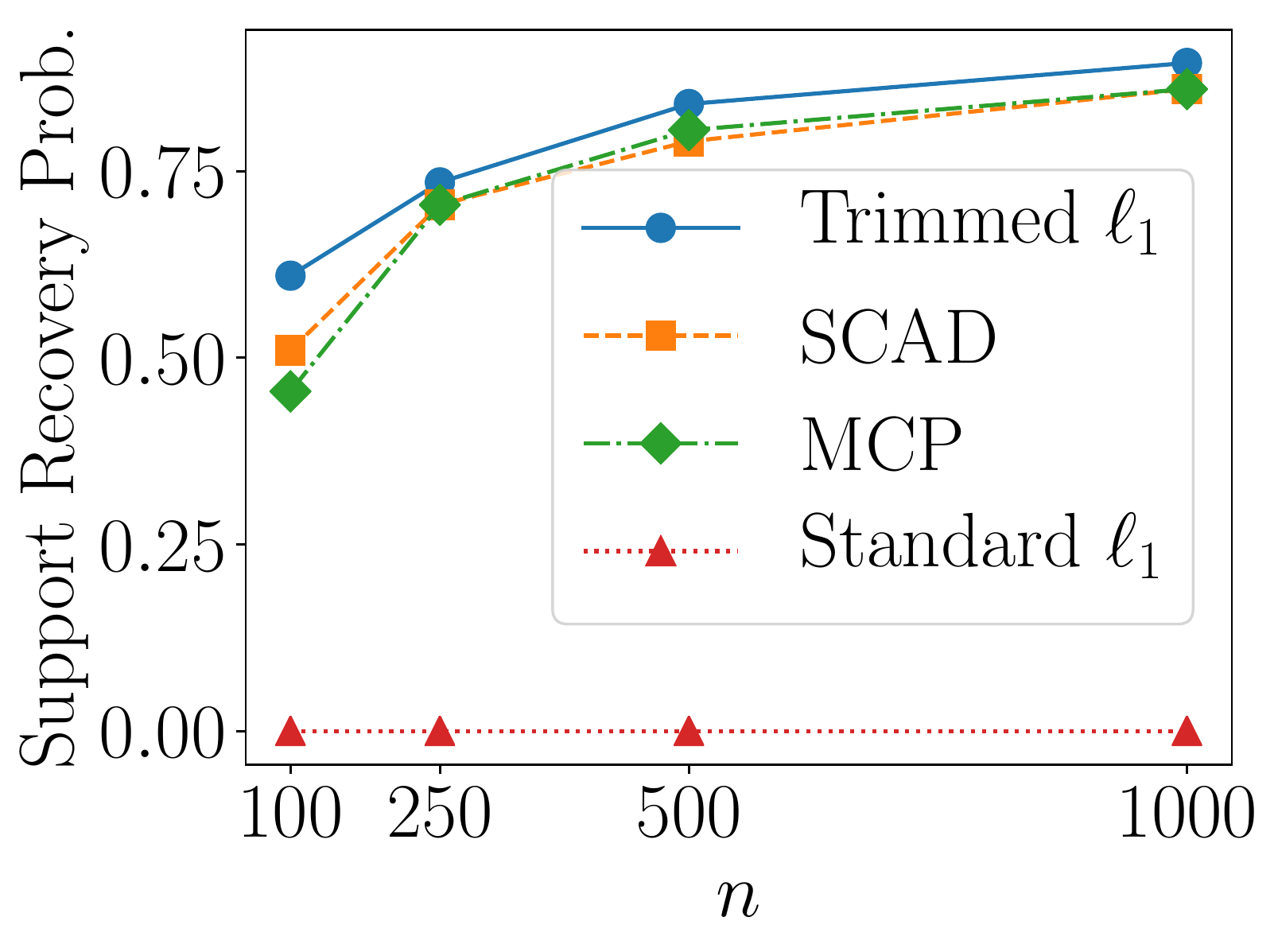}}
		\subfigure[$p=256, k=16$]{\includegraphics[width=0.32\linewidth]{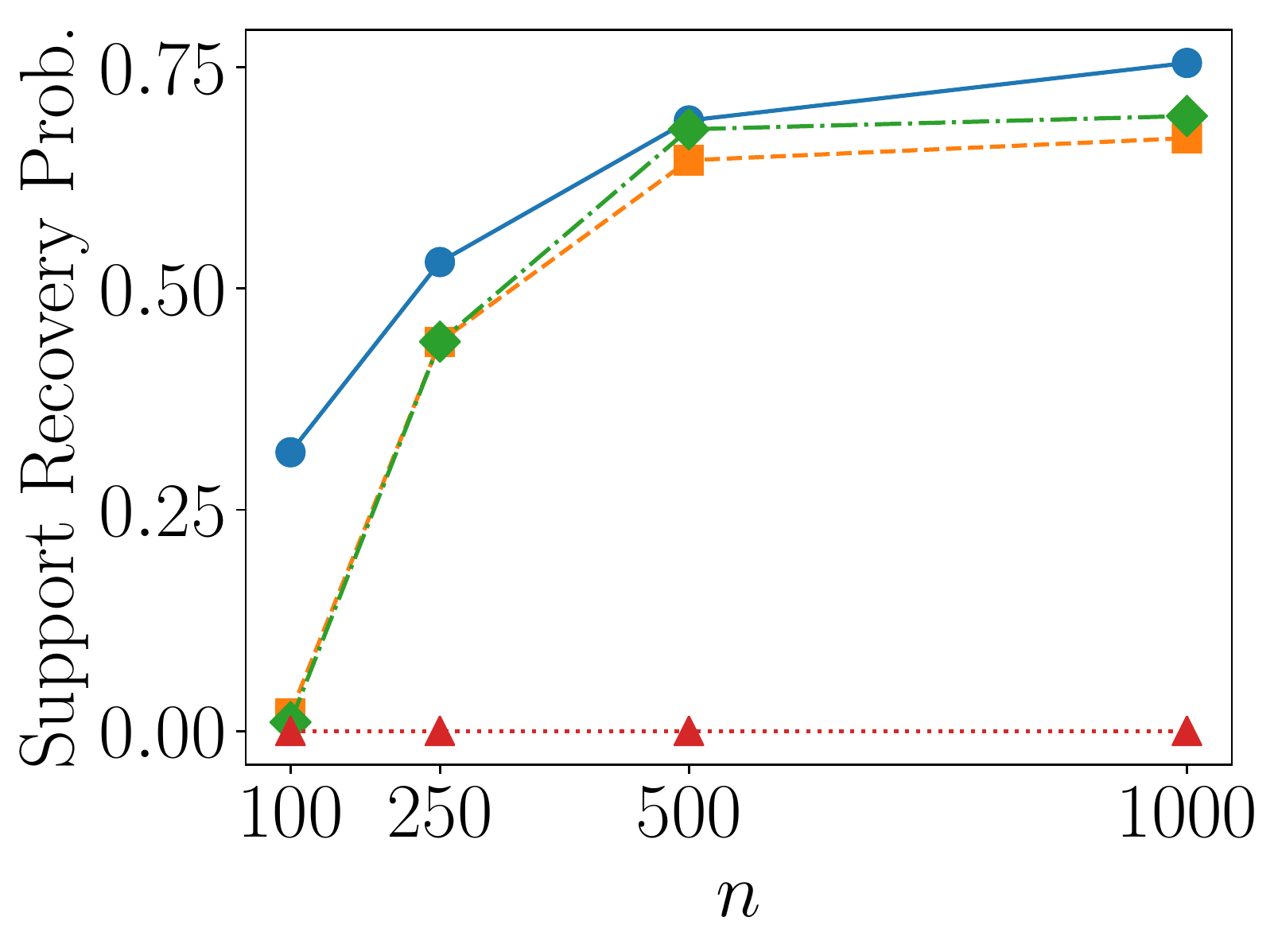}}
		\subfigure[$p=512, k=32$]{\includegraphics[width=0.32\linewidth]{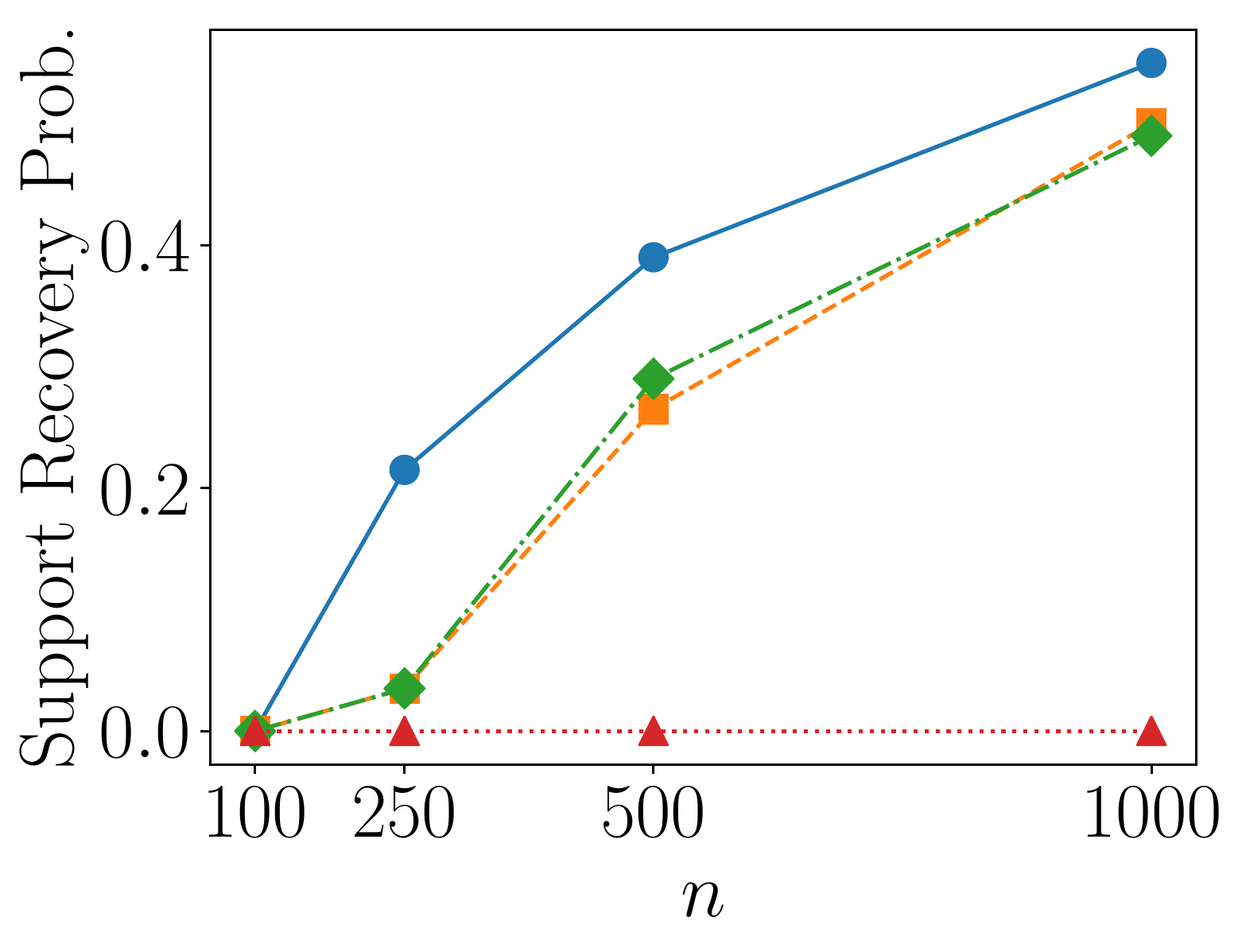}}
		\subfigure[Stationarity]{\includegraphics[width=0.32\linewidth]{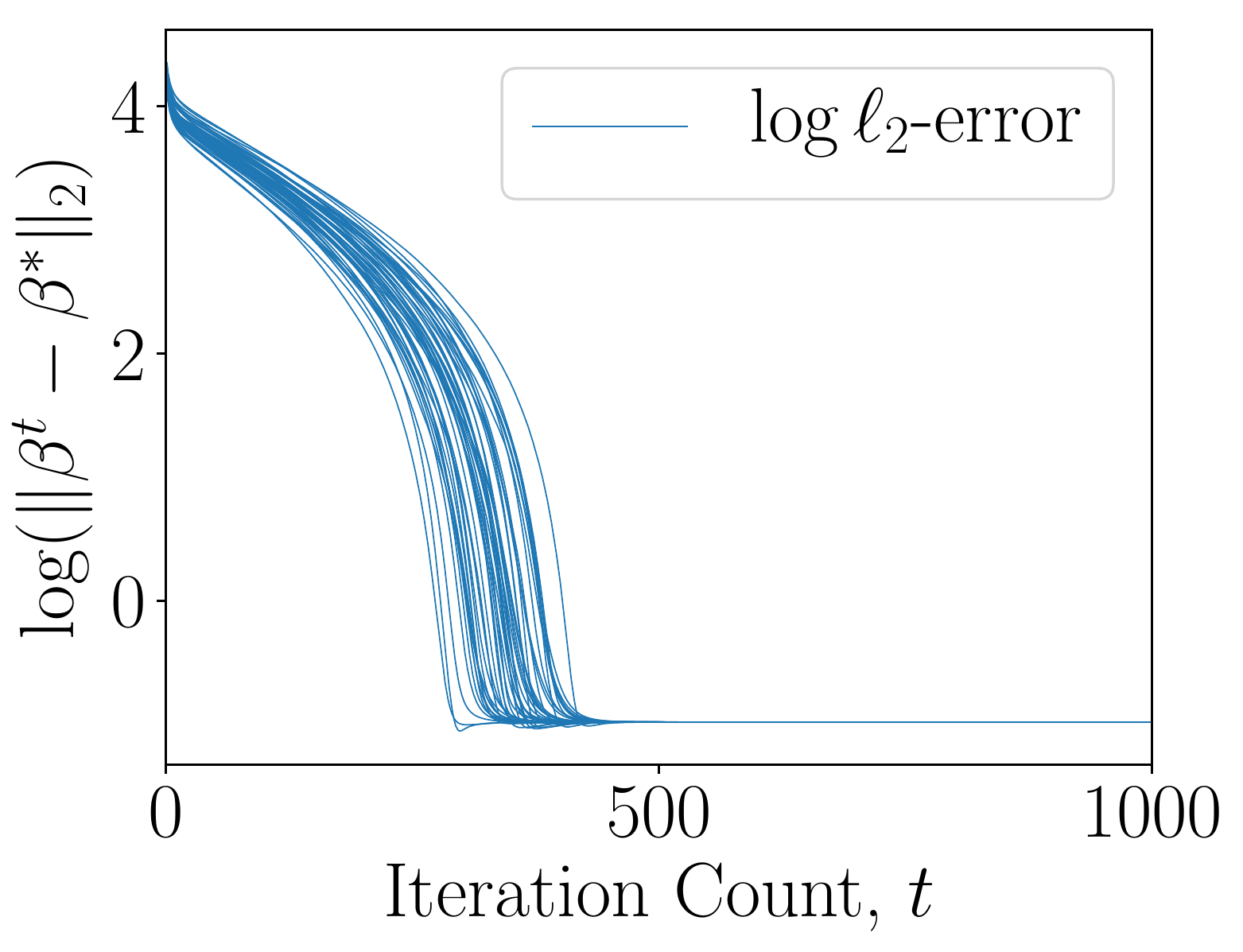}}
		\subfigure[$\log \ell_2$-errors ]{\includegraphics[width=0.32\linewidth]{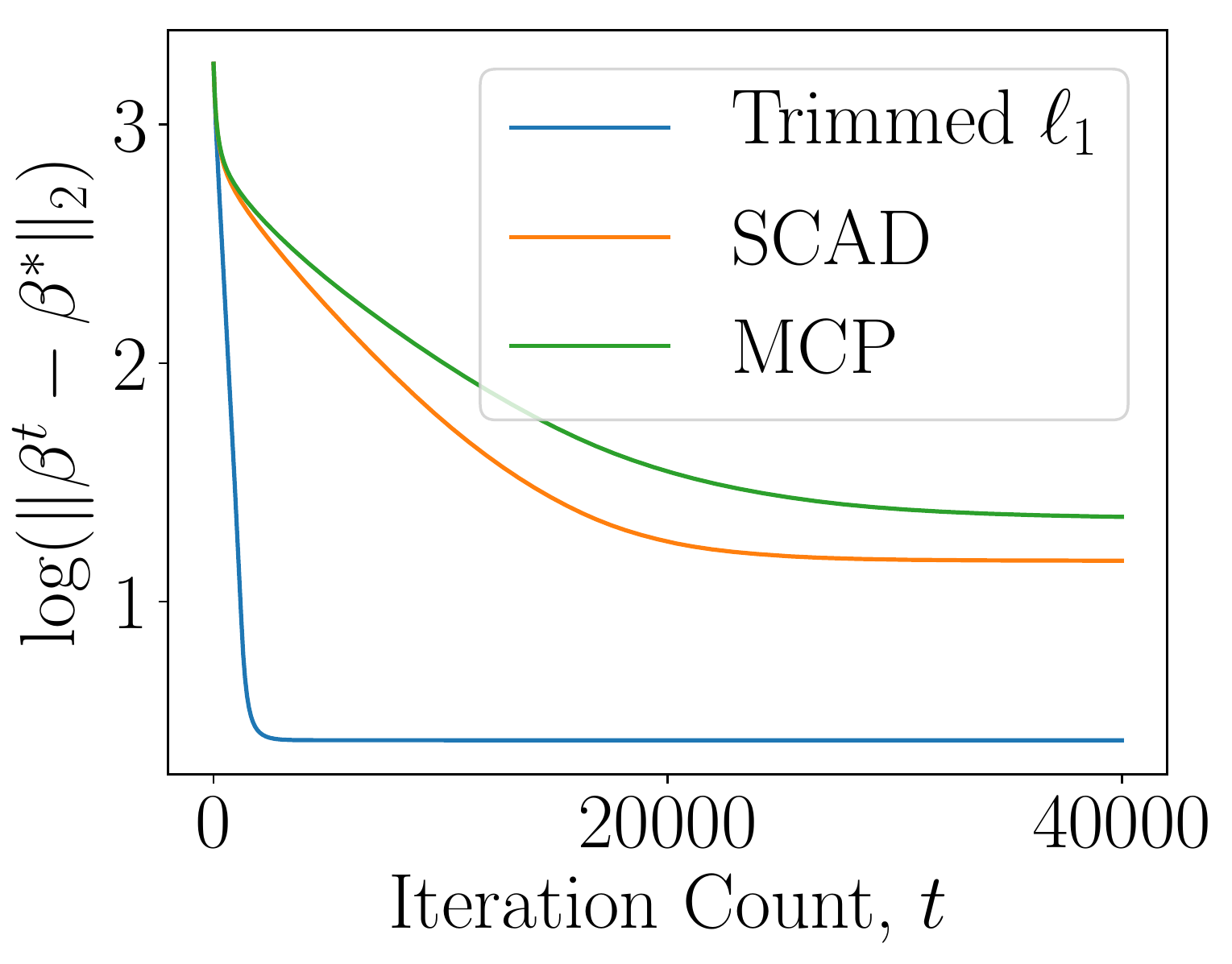}}
		\caption{Results for the non-incoherent case. \textbf{(a)}$\sim$\textbf{(e)}: same as Figure~\ref{Fig:incoherence_reg}.}
		\label{Fig:nonincoherence_reg}
\end{figure}

\begin{figure}[t]
	\centering
\subfigure[Small Regime]{\includegraphics[width=0.32\linewidth]{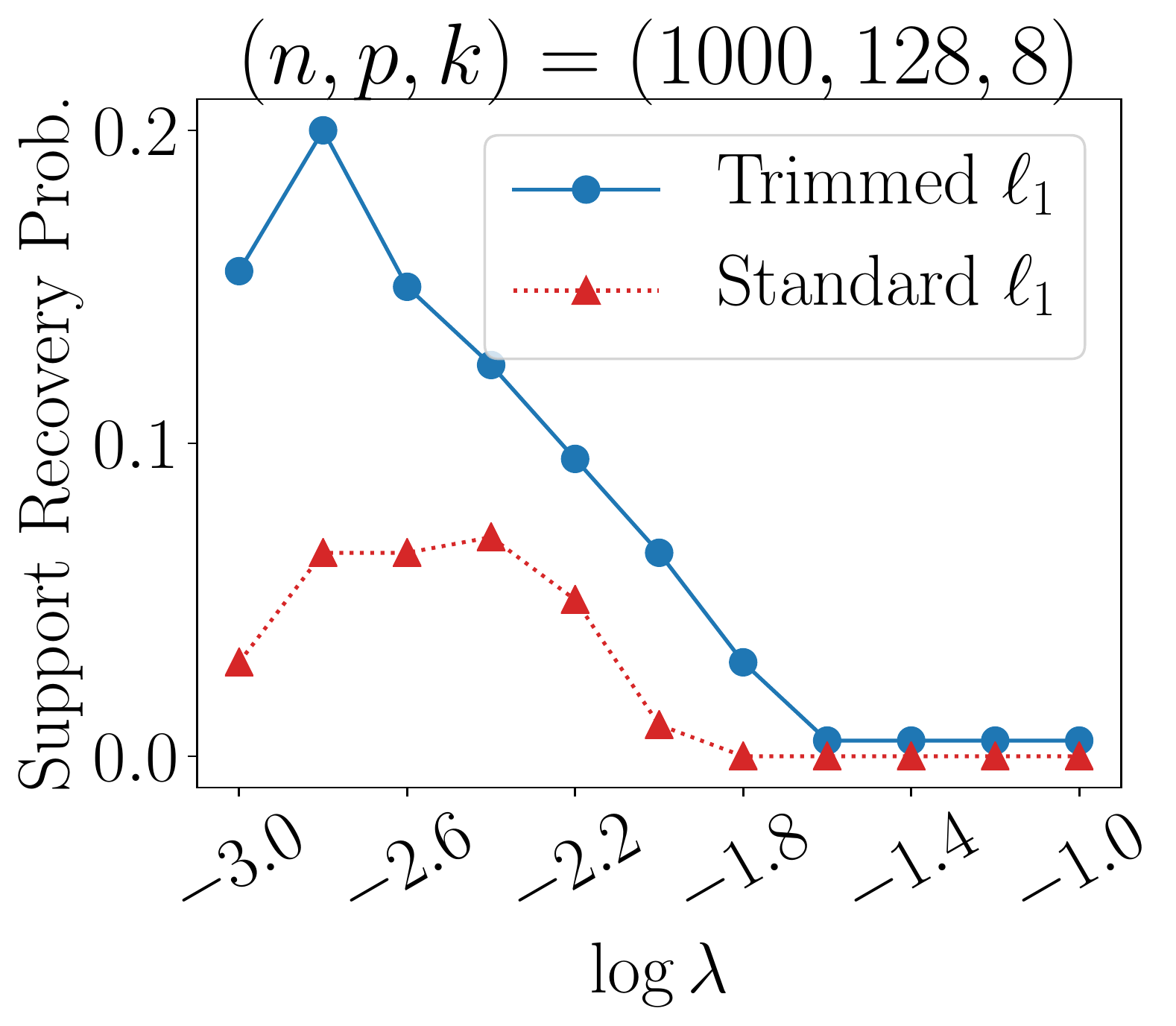}}
\subfigure[Non-incoherent]{\includegraphics[width=0.32\linewidth]{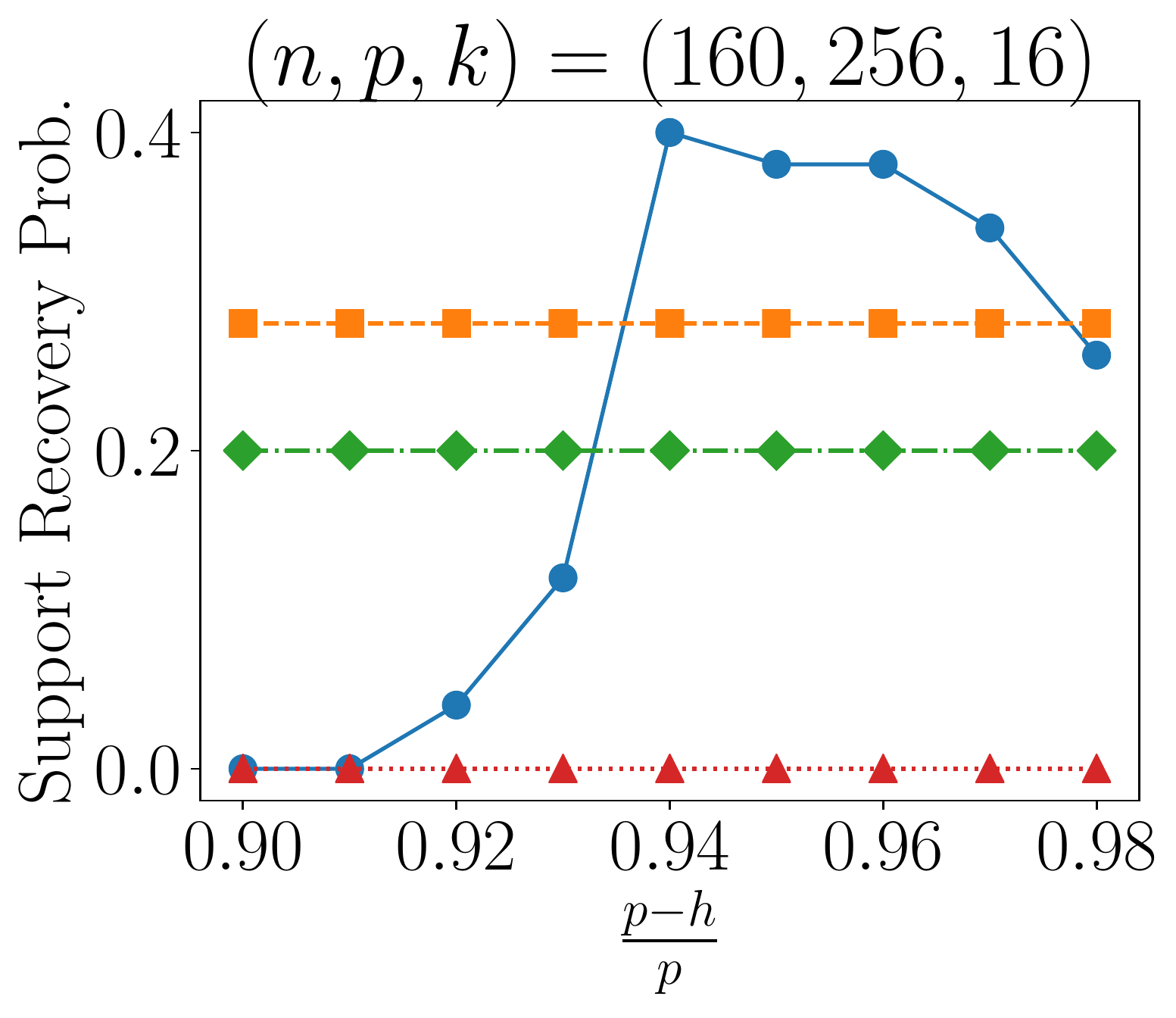}}
\subfigure[Incoherent]{\includegraphics[width=0.32\linewidth]{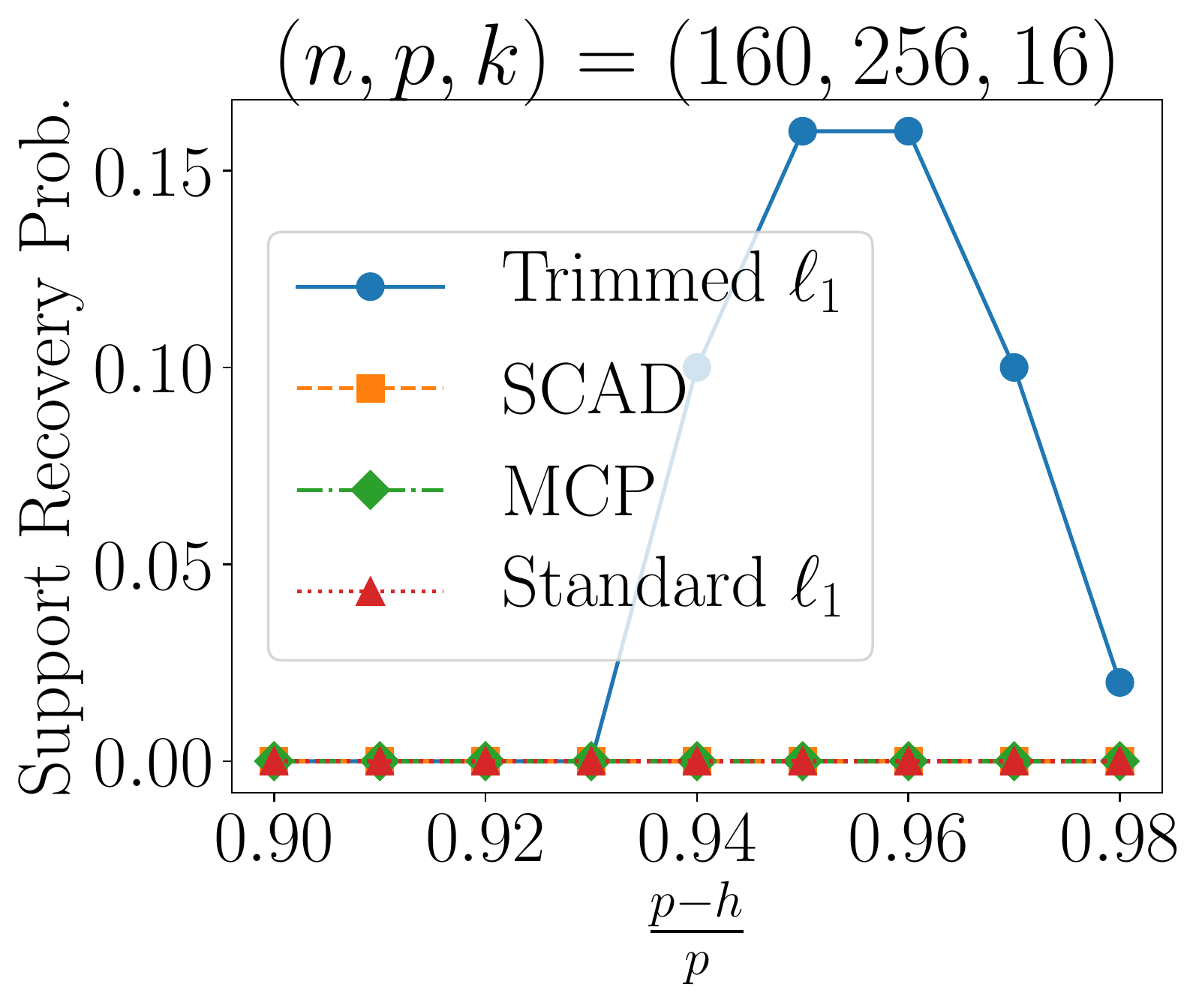}}
\caption{Plots for third and last experiments. \textbf{(a)}: Trimmed Lasso versus standard one in a small regime. We set $h = \lceil 0.05p \rceil$. \textbf{(b)}, \textbf{(c)}: Performance of the trimmed Lasso as the value of $h$ varies.}
\label{Fig:exp_figure}
\end{figure}

We design four experiments. For all experiments except the third one where we investigate the effect of small regularization parameters, we choose the regularization parameters via cross-validation from the set: $\log_{10}\lambda$ $\in$ $\{-3.0, -2.8, \ldots, 1.0\}$. For non-convex penalties requiring additional parameter, we just fix their values (2.5 for MCP and 3.0 for SCAD respectively) since they are not sensitive to results. When we generate feature vectors, we consider two different covariance matrices of normal distribution as introduced in \citet{LW17} to see how regularizers are affected by the incoherence condition.

In our first experiment, we generate i.i.d. observations from $x_i \sim N(0, M_2(\theta))$ where $M_2(\theta) = \theta\mathbf{11}^T+(1-\theta)I_p$ with $\theta$ = 0.7.\footnote{$M_1$ and $M_2$ as defined in \citet{LW17}.} This choice of $M_2(\theta)$ satisfies the incoherence condition~\citet{LW17}. 
We give non-zero values $\beta^*$ with the magnitude sampled from $N(0, 5^2)$, at $k$ random positions, and the response variables are generated by $y_i = x_i^T\beta^* + \epsilon_i$, where $\epsilon_i \sim N(0, 1^2)$. 
In Figure~\ref{Fig:incoherence_reg} (a) $\sim$ (c), we set $(p, k) = (128, 8), (256, 16), (512, 32)$ and increase the sample size $n$. 
The probability of correct support recovery for trimmed Lasso is higher than baselines for all samples in all cases. 
Figure \ref{Fig:incoherence_reg}(d) corroborates Corollary \ref{CorLS}: any local optimum with trimmed $\ell_1$ is close to points with correct support regardless of initialization; 
see comparisons against baselines with same setting in Figure \ref{Fig:incoherence_reg}(e).

In the second experiment, we  replace $M_2(\theta)$ with $M_1(\theta)$, which does not satisfy the incoherence condition.\footnote{$M_1(\theta)$ is a matrix with $1$'s on the diagonal, $\theta$'s in the first $k$ positions of the $(k+1)^{\text{st}}$ row and column, and $0$'s elsewhere.} 
Trimmed still outperforms comparison approaches (Figure~\ref{Fig:nonincoherence_reg}).
Lasso is omitted from Figure \ref{Fig:nonincoherence_reg}(e) as it always fails in this setting.

Our next experiment compares Trimmed Lasso against vanilla Lasso where both $\lambda$ and true non-zeros are small: $\log\lambda \in \{-3.0, -2.8, \ldots, -1.0\}$ and $\beta^* \sim N(0, 0.8^2)$. When the magnitude of $\Tth$ is large, standard Lasso tends to choose a small value of $\lambda$ to reduce the bias of the estimate while Trimmed Lasso gives good performance even for large values of $\lambda$ as long as $h$ is chosen suitably. Figure~\ref{Fig:exp_figure}(a) also confirms the superiority of Trimmed Lasso in a small regime of $\lambda$ with a proper choice of $h$.

In the last experiment, we investigate the effect of choosing the trimming parameter $h$. Figure~\ref{Fig:exp_figure}(b) and (c) show that Trimmed $\ell_1$ outperforms if we set $h = k$ (note $(p-h)/p \approx 0.94$). As $h \downarrow 0$ (when $(p-h)/p = 1$), the performance approaches that of Lasso, as we can see in Corollary \ref{CorLS}.
Additional experiments on sparse Gaussian Graphical Models are provided as supplementary materials.

\paragraph{Input Structure Recovery of Compact Neural Networks.}

We apply the Trimmed $\ell_1$ regularizer to recover input structures of deep models.
We follow~\citet{Oymak18} and consider the regression model $y_i = \bm{1}^T\sigma(\bm{W}^*\bm{x}_i)$ with input dimension $p = 80$, hidden dimension $z = 20$, and ReLU activation $\sigma(\cdot)$. We generate i.i.d. data $\bm{x}_i \sim N(0, I_p)$ and $\bm{W}^* \in \mathbb{R}^{z \times p}$ such that $i$th row has exactly 4 non-zero entries from $N(0, \frac{p}{4z})$ to ensure that $\mathbb{E}[\|\bm{W}^*\bm{x}\|_{\ell_2}^2] = \|\bm{x}\|_{\ell_2}^2$ at only $4(i-1) + 1 \sim 4i$ positions. 
For $\ell_0$ and $\ell_1$ regularizations, we optimize $\bm{W}$ using a projected gradient descent with prior knowledge of $\|\bm{W}^*\|_0$ and $\|\bm{W}^*\|_1$, and we use Algorithm \ref{alg:bcd} for trimmed $\ell_1$ regularization with $h = 4z$ and $(\lambda, \tau) = (0.01, 0.1)$ obtained by cross-validation. 
We set the step size $\eta = 0.1$ for all approaches. 
We consider two sets of simulations with varying sample size $n$ where the initial $\bm{W}_0$ is selected as (a) a small perturbation of $\bm{W}^*$ and (b) at random, as in \citet{Oymak18}. 
Figure~\ref{Fig:neural_net} shows the results where black dots indicate nonzero values in the weight matrix, and we can confirm that Trimmed $\ell_1$ outperforms alternatives in terms of support recovery for both cases.
\begin{figure}[t]
	\centering
	\subfigure[with good initialization (small perturbation from true signal)]
	{\includegraphics[width=\linewidth]{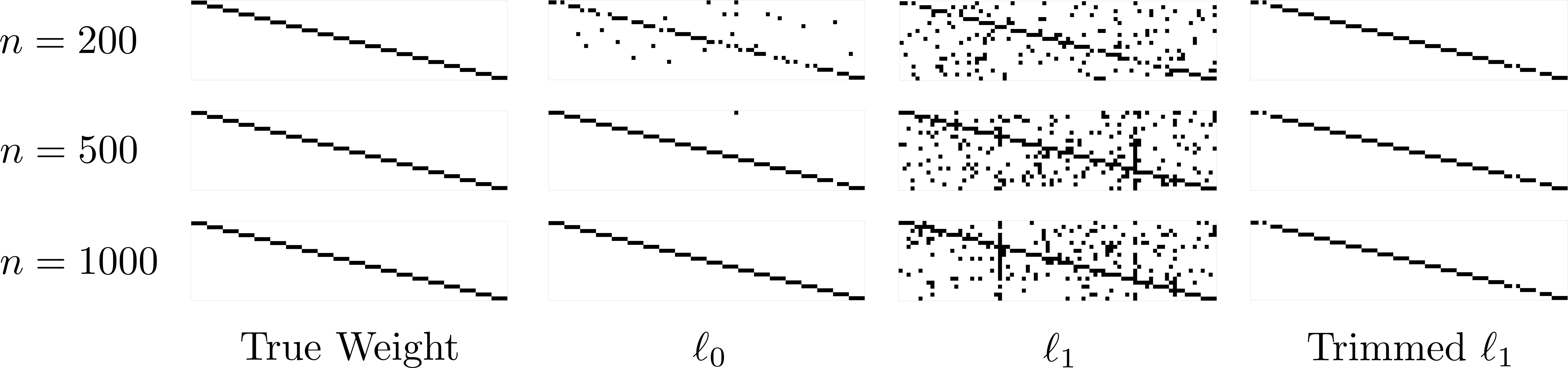}}

	\subfigure[with random initialization]
	{\includegraphics[width=\linewidth]{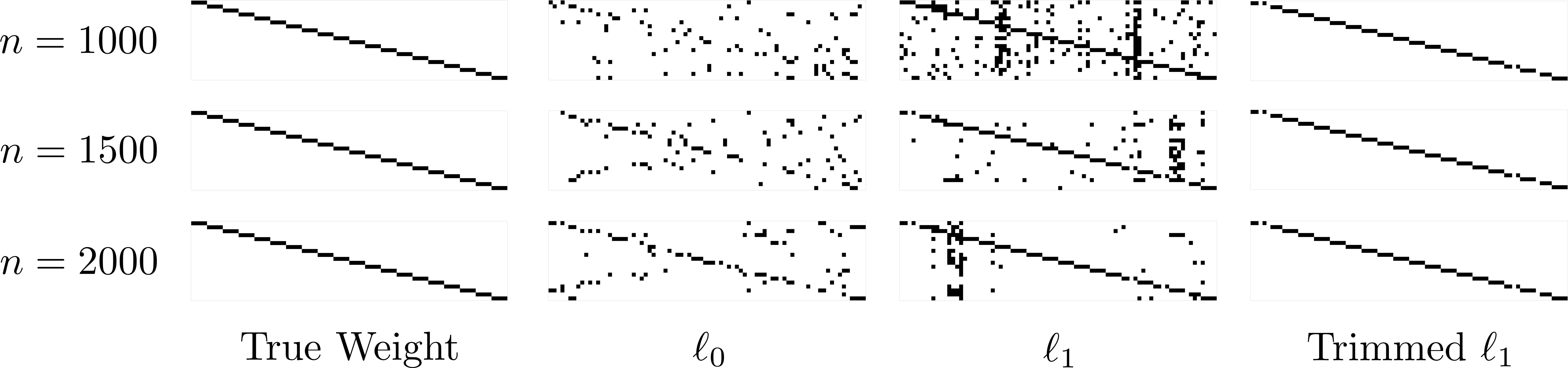}}
	
	\caption{Results for sparsity pattern recovery of deep models.}
	\label{Fig:neural_net}
\end{figure}

\paragraph{Pruning Deep Neural Networks.}

Several recent studies have shown that neural networks are highly over-parameterized, and we can prune the weight parameters/neurons with marginal effect on performance. 
Toward this, we consider trimmed regularization based network pruning. Suppose we have deep neural networks with $L$ hidden layers. Let $n_i$ be the number of neurons in the layer $\bm{h}_i$. The parameters we are interested in are $\mathcal{W} \coloneqq \{\bm{\theta}_l, \bm{b}_l\}_{l=1}^{L+1}$ for $\bm\theta_l \in \mathbb{R}^{n_{l-1} \times n_l}$ and $\bm{b}_l \in \mathbb{R}^{n_l}$ where $\bm{h}_0$ is the input feature $\bm{x}$ and $\bm{h}_{L+1}$ is the output $\bm{y}$. Then, for $l=1,\hdots, L$, $\bm{h}_l = \text{ReLU}(\bm{h}_{l-1} \bm\theta_l + \bm{b}_l)$. Since the edge-wise pruning will not give actual benefit in terms of computation, we prune unnecessary \emph{neurons} through group-sparse encouraging regularizers. Specifically, given the weight parameter $\bm\theta \coloneqq \bm\theta_l$ between $\bm{h}_{l-1}$ and $\bm{h}_l$, we consider the group norm extension of trimmed $\ell_1$:
\[
	\mathcal{R}_l(\bm\theta, \bm{w}) \coloneqq \lambda \sum_{j=1}^{n_{l-1}} w_j \sqrt{\theta_{j,1}^2 + \cdots + \theta_{j,n_l}^2}
\]
with the constraint of $\bm{1}^T\bm{w} = n_{l-1} - h$. %We can set different $h_l$ and $\lambda_l$ for each layer. 
Moreover, we can naturally make an extension to a convolutional layer with encouraging \emph{activation map sparsity} as follows. If $\bm\theta$ is a weight parameter for 2-dimensional convolutional layer (most generally used) with $\bm\theta \in \mathbb{R}^{C_{\text{out}} \times C_{\text{in}} \times H \times W}$, the trimmed regularization term that induces activation map-wise sparsity is given by 
\[
\mathcal{R}_l(\bm\theta, \bm{w}) \coloneqq \lambda \sum_{j=1}^{C_{\text{out}}} w_j \sqrt{\sum_{m,n,k} \theta_{j,m,n,k}^2}
\]
for all possible indices $(m,n,k)$. 
Finally, we add all penalizing terms to a loss function to have 
\[
\mathcal{L}(\mathcal{W}; \mathcal{D}) + \sum_{l=1}^{L+1} \lambda_l \mathcal{R}_l(\bm{\theta}_l, \bm{w}_l)
\]
where we allow different hyperparameters $\lambda_l$ and $h_l$ for each layer.

In Table \ref{tab:naive_comparison}, we compare trimmed group $\ell_1$ regularization against vanilla group $\ell_1$ on MNIST dataset using LeNet-300-100 architecture \citep{Lecun98gradient-basedlearning}. Here, we set the trimming parameter $h$ to half sparsity level of the original model. For the vanilla group $\ell_1$, we need larger $\lambda$ values to obtain sparser models, for which we pay a significant loss of accuracy. In contrast, we can control the sparsity level using trimming parameters $h$ with little or no drop of accuracy.

\begin{table}[t]
	\caption{Results on MNIST using LeNet-300-100.}
	\label{tab:naive_comparison}
	\begin{center}
		\begin{small}
			\begin{tabular}{lcccr}
				\toprule
				Method & Pruned Model & Error ($\%$) \\
				\midrule
				No Regularization & 784-300-100& \textbf{1.6} \\
				grp $\ell_1$ & 784-241-67 & 1.7 \\
				grp $\ell_{1_{\text{trim}}}$, $h = \text{half of original}$ & \textbf{392-150-50} & \textbf{1.6}\\
				\bottomrule
			\end{tabular}
		\end{small}
	\end{center}
\end{table}

\begin{table}[t]
	\caption{Results on MNIST classification for LeNet 300-100 with Bayesian approaches. $h = \circ$ means that the trimming parameter $h$ is set to the same sparsity level of $\circ$, and $\lambda$ sep. indicates that different $\lambda$ values are employed on each layer.}
	\label{tab:l0_comparison}
	\begin{center}
		\begin{small}
			\begin{tabular}{lccr}
				\toprule
				Method & Pruned Model & Error ($\%$) \\
				\midrule
				$\ell_0$ \citep{Louizos18} & 219-214-100 & \textbf{1.4} \\
				$\ell_0$, $\lambda$ sep. \citep{Louizos18} & 266-88-33 & 1.8 \\
				Bayes grp $\ell_{1_{\text{trim}}}$, $h = \ell_0$ & 219-214-100 & \textbf{1.4} \\
				Bayes grp $\ell_{1_{\text{trim}}}$, $h = \ell_0$, $\lambda$ sep. & 266-88-33 & 1.6 \\
				Bayes grp $\ell_{1_{\text{trim}}}$, $h < \ell_0$, $\lambda$ sep. & \textbf{245-75-25} & 1.7 \\
				\bottomrule
			\end{tabular}
		\end{small}
	\end{center}
\end{table}

\begin{table}[t!]
	\caption{Results on MNIST classification for LeNet-5-Caffe with Bayesian approaches.}
	\label{tab:l0_comparison_lenet}
	\begin{center}
		\begin{small}
			\begin{tabular}{lccr}
				\toprule
				Method & Pruned Model & Error ($\%$) \\
				\midrule
				$\ell_0$ \citep{Louizos18} & 20-25-45-462 & \textbf{0.9} \\
				$\ell_0$, $\lambda$ sep. \citep{Louizos18} & 9-18-65-25 & 1.0 \\
				Bayes grp $\ell_{1_{\text{trim}}}$, $h < \ell_0$ & 20-25-45-150 & \textbf{0.9} \\
				Bayes grp $\ell_{1_{\text{trim}}}$, $h = \ell_0$, $\lambda$ sep. & 9-18-65-25 & 1.0 \\
				Bayes grp $\ell_{1_{\text{trim}}}$, $h < \ell_0$, $\lambda$ sep. & \textbf{8-17-53-19} & 1.0 \\
				\bottomrule
			\end{tabular}
		\end{small}
	\end{center}
\end{table}

Most algorithms for network pruning recently proposed are based on a variational Bayesian approach \cite{dai2018vib, Louizos18}. Motivated by learning sparse structures via smoothed version of $\ell_0$ norm ~\citep{Louizos18}, we propose a Bayesian neural network with trimmed regularization where we regard only $\bm\theta$ as Bayesian. Inspired by a relation between variational dropout and Bayesian neural networks ~\cite{Kingma15}, we specifically choose a fully factorized Gaussian as a variational distribution, $q_{\bm\phi, \bm\alpha}(\theta_{i,j}) = \mathcal{N}(\phi_{i,j}, \alpha_{i,j}\phi_{i,j}^2)$, to approximate the true posterior and leave $\bm{w}$ to directly learn sparsity patterns. Then the problem is cast to maximizing corresponding evidence lower bound (ELBO),
\[
	\mathbb{E}_{q_{\bm\phi, \bm\alpha}}[\mathcal{L}(\mathcal{W};\mathcal{D})] - \mathbb{KL}\big(q_{\bm\phi, \bm\alpha}(\mathcal{W}) \| p(\mathcal{W})\big).
\]
Combined with trimmed $\ell_1$ regularization, the objective is
\begin{equation}
\label{bayesian_loss}
\begin{aligned}
	\mathbb{E}_{q_{\bm\phi, \bm\alpha}(\bm\theta)}\Big[-\mathcal{L}(\mathcal{W}; \mathcal{D}) + \sum\limits_{l=1}^{L+1} \lambda_l \mathcal{R}_l(\bm\theta_l, \bm{w}_l)\Big] \\
	+~ \mathbb{KL}(q_{\bm\phi, \bm\alpha}(\mathcal{W}) \| p(\mathcal{W}))
\end{aligned}
\end{equation}
which can be interpreted as a sum of expected loss and expected trimmed group $\ell_1$ penalizing term. 
\citet{Kingma14} provide the efficient unbiased estimator of stochastic gradients for training $(\bm\phi, \bm\alpha)$, via the reparameterization trick to avoid computing gradient of sampling process. In order to speed up our method, we approximate expected loss term in \eqref{bayesian_loss} using a local reparameterization trick ~\citep{Kingma15} while the standard reparameterization trick is used for the penalty term.

 Trimmed group $\ell_1$ regularized Bayesian neural networks have smaller capacity with less error than other baselines (Table \ref{tab:l0_comparison}). 
 Our model has lower error rate and better sparsity even for convolutional network, LeNet-5-Caffe\footnote{\url{https://github.com/BVLC/caffe/tree/master/examples/mnist}} (Table \ref{tab:l0_comparison_lenet}).\footnote{We only consider methods based on sparsity encouraging regularizers. State-of-the-art VIBNet ~\citep{dai2018vib} exploits the mutual information between each layer.} 
 
 The code is available at \url{https://github.com/abcdxyzpqrst/Trimmed_Penalty}.

%%%%%%% CONCLUSION %%%%%%%
\section{Concluding Remarks}
In this work we studied statistical properties of high-dimensional $M$-estimators with the trimmed $\ell_1$ penalty, and demonstrated 
the value of trimmed regularization compared to convex and non-convex alternatives. 
We developed a provably convergent algorithm for the trimmed problem,  
based on specific problem structure rather than generic DC structure, with promising numerical results. 
A detailed comparison to DC based approaches is left to future work. Going beyond $M$-estimation, we showed that trimmed regularization can be beneficial for two deep learning tasks: input structure recovery and network pruning.
As future work we plan to study trimming of general decomposable regularizers, including $\ell_1 / \ell_q$  norms, and further investigate the use of trimmed regularization in deep models.

{\small \paragraph{Acknowledgement.}
This work was supported by  the National Research Foundation of Korea (NRF) grant (NRF-2018R1A5A1059921), Institute of Information \& Communications Technology Planning \& Evaluation (IITP) grant (No.2019-0-01371) funded by the Korea government (MSIT) and Samsung Research Funding \& Incubation Center via SRFC-IT1702-15.}

\newpage
\bibliography{TrimReg,sml,Elem_LinearModel,opt,neuralnet}

\newpage

%%%%%%% APPENDIX %%%%%%%
%
%\appendix
\section*{Supplementary Materials}

\section{Sparse graphical models}

We derive a corollary for the trimmed Graphical lasso:
\begin{align}\label{EqnGlasso}
		\minimize_{\Th \in \mathcal{S}^{p}_{++}} \ \textrm{trace}\big(\Sig \Th \big) -\log\det \big(\Th\big) + \lam \R(\Th_{\textrm{off}};h)  .	
	\end{align}
Following the strategy of \cite{LW17}, we assume that the sample size scales with the row sparsity $d$ of true inverse covariance $\TTh = (Cov(X))^{-1}$, which is a milder condition than other works ($n$ scaling with $k$, the number of non zero entries of $\TTh$):
\begin{corollary}\label{CorGlasso}
	Consider the program \eqref{EqnGlasso} where the $x_i$'s are drawn from a sub-Gaussian and sample size $n > c_0 d^2 \log p$ with the selection of  	
	\vspace{-.2cm}\begin{enumerate}[leftmargin=0.25cm, itemindent=0.45cm,label=(\alph*)]
	    \item $\lam \geq \cLStwo \sqrt{\frac{\log p}{n}}$ for some constant $\cLStwo$ depending only on $\TTh$
 		\item $h$ satisfying: for any selection of $\Nonreg \subseteq \{1,2,\hdots,p\} \times \{1,2,\hdots,p\} \text{ s.t. } |\Nonreg| = h$,
	 		\begin{align}\label{EqnGlassoIncoh}
				& \matnormBig{\big(\TTh \otimes \TTh\big)_{\USupNonreg \USupNonreg}}{\infty}  \leq \cLSone,  \\
				& \max\left\{\matnorm{\GG_{\USupNonregC \USupNonregC}}{\infty},\matnorm{(\GG_{\USupNonreg \USupNonreg})^{-1}}{\infty}\right\} \leq \cLSsix \quad \text{ and } \nonumber\\
				& \matnormBig{\big({\TTh}^{-1} \otimes {\TTh}^{-1}\big)_{\USupNonregC\USupNonreg} \Big(\big({\TTh}^{-1} \otimes {\TTh}^{-1}\big)_{\USupNonreg\USupNonreg}\Big)^{-1}}{\infty} \leq \eta \nonumber.
			\end{align}
 	\end{enumerate}
 	Further suppose that $\frac{1}{2}\TTh_{\min}$ is lower bounded by $\cLSthree \sqrt{\frac{\log p}{n}} + 2 \lam \cLSone$ for some constant $\cLSthree$.
 	Then with high probability at least $1-\cLSfour \exp (-\cLSfive \log p)$, any local minimum $\LTh$ of \eqref{EqnLS} has the following property: 
 	\begin{enumerate}[leftmargin=0.25cm, itemindent=0.45cm,label=(\alph*)]
 		\item For every pair $j_1 \in \Supp, j_2 \in \SuppC$, $|\LTh_{j_1}| > |\LTh_{j_2}|$,
 		\item If $h < k$, all $j \in \SuppC$ are successfully estimated as zero and we have
 			\begin{align}
 				\|\LTh -\TTh\|_\infty \leq \cLSthree \sqrt{\frac{\log p}{n}} + 2\lam \cLSone
 			\end{align}
 		\item If $h \geq k$, at least the smallest $p-h$ entries in $\SuppC$ have exactly zero and we have
 			\begin{align}
				  \|\LTh -\TTh\|_\infty \leq \cLSthree \sqrt{\frac{\log p}{n}}.
 			\end{align}
  	\end{enumerate}
\end{corollary}
Note that condition \eqref{EqnGlassoIncoh} is the incoherence condition studied in \cite{RWRY11}, and the same remarks as those for sparse linear models (see Section~\ref{sec:slr}) can be made. %above and comparable to counterparts for $\ell_1$ or $(\mu,\gamma)$ regularized Glasso \citep{LW17}. 

\section{Proofs}
\subsection{Proof of Theorem \ref{ThmSupp}}
We extend the standard PDW technique \cite{Wainwright2006new,YRAL13,LW17} for the trimmed regularizers. For any {\bf fixed} $\Nonreg$, we construct a primal and dual witness pair with the strict dual feasibility. Specifically, given the fixed $\Nonreg$, consider the following program:
	\begin{align}\label{EqnThmSup1}
		\minimize_{\th \in \Omega} \ \ &  \L(\th;\Data) + \lam \sum_{j \in \NonregC} |\theta_j| .	
	\end{align}
	Note that the program \eqref{EqnThmSup1} is convex (under \ref{Con:diff}) where the regularizer is only effective over entries in (fixed) $\NonregC$. We construct the primal and dual pair $(\Gth,\Gz)$ by the following restricted program
	\begin{align}\label{EqnRest}
		\Gth \in \argmin_{\th \in \reals^{\USupNonreg} : \ \th \in \Omega} \ \ &  \L(\th) + \lam \R(\th;h) 
	\end{align}
	 and \eqref{EqnPDW}. The following lemma can guarantee under the strict dual feasibility that any solution of \eqref{EqnThmSup1} has the same sparsity structure on $\NonregC$ with $\Gth$. Moreover, since the restricted program \eqref{EqnPDW} is strictly convex as shown in the lemma below, we can conclude that $\Gth$ is the unique minimum point of the restricted program \eqref{EqnThmSup1} given $\Nonreg$.     

	\begin{lemma}\label{LemUnique}
		Suppose that there exists a primal optimal solution $\Gth$ for \eqref{EqnThmSup1} with associated sub-gradient (or dual) $\Gz$ such that $\|\Gz_{\USupNonregC}\|_\infty < 1$. Then any optimal solution $\Lth$ of \eqref{EqnThmSup1} will satisfy $\Lth_j = 0$ for all $j \in \USupNonregC$.  	
	\end{lemma}
	\begin{proof}
		The lemma can be directly achieved by the basic property of convex optimization problem, as developed in existing works using PDW \cite{Wainwright2006new,YRAL13}. Note that even though the original problem with the trimmed regularizer is not convex, \eqref{EqnThmSup1} given $\Nonreg$ is convex. Therefore, by complementary slackness, we have $\sum_{j\in \NonregC} |\Lth_j| = \langle \Gz_{\NonregC}, \Lth_{\NonregC} \rangle$. Therefore, any optimal solution of \eqref{EqnThmSup1} will satisfy $\Lth_j = 0$ for all $j \in \USupNonregC$ since the associated (absolute) sub-gradient is strictly smaller than 1 by the assumption in the statement.   	
	\end{proof}
	
	\begin{lemma}[Section A.2 of \citep{LW17}]\label{LemInv}
		Under \ref{Con:rsc}, the loss function $\L(\th)$ is strictly convex on $\th \in \reals^{\USupNonreg}$ and hence $\big( \nabla^2 \L(\th)\big)_{\USupNonreg\USupNonreg}$ is invertible if $n \geq \frac{2\RSCtolOne}{\RSCcon} (k+h)\log p$.
	\end{lemma}

Now from the definition of $\Q$, we have 
	\begin{align}\label{EqnThm1Q}
		\Q (\Gth - \Tth) = \nabla\L(\Gth) - \nabla\L(\Tth)	
	\end{align}
	where $\Q$ is decomposed as 
	$\begin{bmatrix}
    	\Q_{\USupNonreg\USupNonreg} & \Q_{\USupNonreg\USupNonregC} \\
    	\Q_{\USupNonregC\USupNonreg} & \Q_{\USupNonregC\USupNonregC} \\
    \end{bmatrix}$.
    Then by the invertibility of $\big( \nabla^2 \L(\th)\big)_{\USupNonreg\USupNonreg}$ in Lemma \ref{LemInv} and the zero sub-gradient condition in \eqref{EqnPDW} we have
    \begin{align}
    	\Gth_{\USupNonreg} - \Tth_{\USupNonreg} = \Big(\Q_{\USupNonreg\USupNonreg}\Big)^{-1} \Big( - \nabla\L(\Tth)_{\USupNonreg} - \lam \Gz_{\USupNonreg} \Big).
    \end{align}
 
	Since both $\Gth_{\USupNonregC}$ and $\Tth_{\USupNonregC}$ are zero vectors, we obtain  
	\begin{align}\label{EqnThmSup2}
		\| \Gth - \Tth \|_\infty  \ & = \Big\| \Big(\Q_{\USupNonreg\USupNonreg}\Big)^{-1} \Big( - \nabla\L(\Tth)_{\USupNonreg} - \lam \Gz_{\USupNonreg} \Big)	\Big\|_\infty \nonumber\\
		& \leq \Big\| \Big(\Q_{\USupNonreg\USupNonreg}\Big)^{-1} \nabla\L(\Tth)_{\USupNonreg} \Big\|_\infty + \lam \matnormbig{\big(\Qs\big)^{-1}}{\infty}. 
	\end{align}
	Therefore, under the assumption on $\Tth_{\min}$ in the statement, the selection of $\Nonreg$ in which there exists some $(j,j')$ s.t. $j\in \Supp$, $j \in \NonregC$, $j' \in \SuppC$ and $j' \in \Nonreg$, yields contradictory solution with \eqref{EqnTrimmedReg2}. Under the strict dual feasibility condition for this specific choice of $\Nonreg$ (along with Lemma \ref{LemInv}) can guarantee that there is no local minimum for that choice of $\Nonreg$.   
	 Hence, \eqref{EqnThmSup2} can guarantee that for every pair $(j_1,j_2)$ such that $j_1 \in \Supp$ and $j_2 \notin \Supp$, we have $|\Lth_{j_1}| > |\Lth_{j_2}|$ (since $\Lth = \Gth$). Note that for any \emph{valid} selection of $\Nonreg$, this statement holds. This immediately implies that any local minimum of \eqref{EqnTrimmedReg2} satisfies this property as well, as in the statement. 
	
	Finally turning to the bound when $h \geq k$, we have $\USupNonreg = \Nonreg$ since all entries in $\Supp$ are not penalized as shown above. In this case, $\Gz_{\USupNonreg}$ becomes zero vector (since $\DSupNonreg$ is empty in the construction of $\Gz$), and the bound in \eqref{EqnThmSup2} will be tighter as 
	\begin{align}\label{EqnThmSup2}
		\| \Gth - \Tth \|_\infty  \ & = \Big\| \Big(\Q_{\USupNonreg\USupNonreg}\Big)^{-1} \Big( - \nabla\L(\Tth)_{\USupNonreg} - \lam \Gz_{\USupNonreg} \Big)	\Big\|_\infty \nonumber\\
		& \leq \Big\| \Big(\Q_{\USupNonreg\USupNonreg}\Big)^{-1} \nabla\L(\Tth)_{\USupNonreg} \Big\|_\infty ,
	\end{align}
	as claimed.

%\subsection{Proof of Theorem \ref{ThmL2}}
\subsection{Proof of Theorem \ref{ThmL2}}
Here we adopt the strategy developed in \citet{YLA2018} for analyzing local optima of trimmed loss function. Since our loss function $\L$ is convex, the story derived in this subsection can also be applied to results of \citet{NRWY12}. However, in order to simplify the procedure, we will not utilize the convexity of $\L$ and instead place the side constraint $\|\th \|_1 \leq R$ and some additional assumptions (see \cite{Loh13} for details). As in \citet{YLA2018}, we introduce the the shorthand to denote local optimal error vector: $\Lerrt := \Lth - \Tth$ given an \emph{arbitrary} local minimum $(\Lth,\Lw)$ of \eqref{EqnTrimmedReg2}. We additionally define $H$ to denote the set of indices not penalized by $\Lw$ (that is, $\Lw_j = 0$ for $j \in H$, $\Lw_j = 1$ for $j \in H^c$ and $|H| = h$). Utilizing the fact that $(\Lth,\Lw)$ is a local minimum of \eqref{EqnTrimmedReg2}, we have an inequality 
\begin{align*}
	\biginner{\nabla_{\th} \L\big(\Tth + \Lerrt \big)}{\Lth - \th}  \leq - \biginner{  \partial \lambda \R(\Tth + \Lerrt; h)}{\Lth - \th} \, \quad \text{for any feasible } \th.
\end{align*}
This inequality comes from the first order stationary condition (see \citet{LW15} for details) in terms of $\th$ fixing $\w$ at $\Lw$. Here, if we take $\th = \Tth$ above, we have
\begin{align*}
	& \biginner{\nabla \L\big(\Tth + \Lerrt \big)}{\Lerrt} \leq - \biginner{\partial \lambda \R(\Tth + \Lerrt;h)}{\Lerrt} \overset{(i)}{\leq}  \lambda( \|\Tth_{H^c}\|_1 - \|\Lth_{H^c}\|_1 )
\end{align*}
where $S$ is true support set of $\Tth$ and the inequality $(i)$ holds due to the convexity of $\ell_1$ norm. 

\paragraph{i) $h < k$: } By Theorem \ref{ThmSupp}, we can guarantee with high probability that $H \subset S$. Then, by triangular inequality (in inequality $(ii)$ below) and the fact that $\Tth$ is $S$-sparse vector, we have   
\begin{align}\label{EqnLocalMinima}
	& \biginner{\nabla \L\big(\Tth + \Lerrt \big)}{\Lerrt} \leq \,  \lambda( \|\Tth_{H^c}\|_1 + \|\Lerrt_{S^c}\|_1  - \|\Lerrt_{S^c}\|_1  - \|\Lth_{H^c}\|_1 ) = \lambda( \|\Tth_{H^c} + \Lerrt_{S^c}\|_1  - \|\Lerrt_{S^c}\|_1  - \|\Lth_{H^c}\|_1 )  \nonumber\\
	\overset{(ii)}{\leq} \, & \lambda\big( \|\Tth_{H^c} + \Lerrt_{S^c} + \Lerrt_{S-H}\|_1 + \|\Lerrt_{S-H}\|_1 - \|\Lerrt_{S^c}\|_1  - \|\Lth_{H^c}\|_1 \big) = \lambda(\|\Lerrt_{S-H}\|_1 - \|\Lerrt_{S^c}\|_1 ) \, .
\end{align}
Combining \eqref{EqnLocalMinima} and \ref{Con:rsc} yields 
\begin{align*}
	& \RSCcon \|\Lerrt\|_2^2  - \RSCtolOne \frac{\log p}{n}\|\Lerrt\|_1^2 \leq \biginner{\nabla \L\big(\Tth + \Lerrt\big) - \nabla\L\big(\Tth\big) }{\Lerrt}
	\leq  - \biginner{\nabla \L \big(\Tth\big) }{\Lerrt} + \lambda\, \big(\|\Lerrt_{S-H}\|_1 - \|\Lerrt_{S^c}\|_1 \big).
\end{align*}
If we assume $\max \big\{ \| \nabla \L\big(\Tth\big) \|_\infty, 2\rho\RSCtolOne\frac{\log p}{n} \big\} \leq \frac{\lambda}{4}$ (which are slightly different to assumptions in the statement, however they are purely for simplicity and can be relaxed if we use the convexity of $\L$, as we mentioned in the beginning of the proof), we can conclude that 
\begin{align}\label{EqnMainTmp2}
	0 \leq \, & \RSCcon \|\Lerrt\|_2^2  \leq  \big\| \nabla \L \big(\Tth\big) \big\|_\infty \|\Lerrt\|_1 + \lambda\, \big(\|\Lerrt_{S-H}\|_1 - \|\Lerrt_{S^c}\|_1 \big) + 2 \rho \RSCtolOne \frac{\log p}{n} \|\Lerrt\|_1 \nonumber\\
	\leq \, & \frac{\lambda}{2}\|\Lerrt\|_1 - \lambda \|\Lerrt_{S^c}\|_1  + \lambda \|\Lerrt_{S-H}\|_1 \leq \frac{\lambda}{2}\|\Lerrt_{S}\|_1 - \frac{\lambda}{2}\|\Lerrt_{S^c}\|_1  + \lambda \|\Lerrt_{S-H}\|_2 \leq \frac{\lambda}{2}\|\Lerrt_{S}\|_1  + \lambda \|\Lerrt_{S-H}\|_2 .
\end{align}
As a result, we can finally have an $\ell_2$ error bound as follows:
\begin{align*}
	& \RSCcon \|\Lerrt\|_2^2  \leq \frac{\lambda\sqrt{k}}{2}\|\Lerrt_{S}\|_2  + \lambda\sqrt{k-h} \|\Lerrt_{S-H}\|_2 \leq \left( \frac{\lambda\sqrt{k}}{2}  + \lambda\sqrt{k-h}\right) \|\Lerrt \|_2 
\end{align*}
implying that 
\begin{align*}
	\|\Lerrt\|_2 \leq \frac{1}{\RSCcon}\left( \frac{\lambda\sqrt{k}}{2}  + \lambda\sqrt{k-h}\right)  \, .
\end{align*}

\paragraph{ii) $h \geq k$: } As in the previous case, Theorem \ref{ThmSupp} can guarantee $S \subseteq H$ where equality holds if $h=k$. Instead of \eqref{EqnLocalMinima}, now we have
\begin{align}\label{EqnLocalMinimaCase2}
	& \biginner{\nabla \L\big(\Tth + \Lerrt \big)}{\Lerrt} \leq \,  \lambda( \|\Tth_{H^c}\|_1 + \|\Lerrt_{H^c}\|_1  - \|\Lerrt_{H^c}\|_1  - \|\Lth_{H^c}\|_1 ) \nonumber\\
	= \ & \lambda( \|\Tth_{H^c} + \Lerrt_{H^c}\|_1  - \|\Lerrt_{H^c}\|_1  - \|\Lth_{H^c}\|_1 ) = - \|\Lerrt_{H^c}\|_1.
\end{align}
By similar reasoning in the case of i), we combine \eqref{EqnLocalMinimaCase2} and \ref{Con:rsc} to obtain
\begin{align}\label{EqnMainTmp2}
	0 \leq \, & \RSCcon \|\Lerrt\|_2^2  \leq \frac{\lambda}{2}\|\Lerrt\|_1 - \lambda \|\Lerrt_{H^c}\|_1   \leq \frac{\lambda}{2}\|\Lerrt_{H}\|_1 - \frac{\lambda}{2}\|\Lerrt_{H^c}\|_1  \leq \frac{\lambda}{2}\|\Lerrt_{H}\|_1 \leq \frac{\lambda\sqrt{h}}{2}\|\Lerrt\|_2 
\end{align}
implying that
\begin{align*}
	\|\Lerrt\|_2 \leq \frac{1}{\RSCcon} \frac{\lambda\sqrt{h}}{2} \, .
\end{align*}

\subsection{Proof of Corollary \ref{CorLS}}
The proof our corollary is similar to that of Corollary 1 of \cite{LW17}, who derive the result for $(\mu,\gamma)$-amenable regularizers. Here we only describe the parts that need to be modified from \cite{LW17}.

In order to utilize theorems in the main paper, we need to establish the RSC condition \ref{Con:rsc} and the strict dual feasibility:
			%\begin{align}\label{EqnStrictDual}
				$\|\Gz_{\USupNonregC}\|_\infty \leq 1 - \delta$
			%\end{align}

 First, the RSC is known to hold w.h.p as shown in several previous works such as Lemma \ref{LemLSRSC}.

 \begin{lemma}[Corollary 1 of \cite{LW15}]\label{LemLSRSC}
 	The RSC condition in \ref{Con:rsc} for linear models holds with high probability with $\RSCcon=\frac{1}{2}\lambda_{\min}(\Sigma_{x})$ and $\RSCtolOne \asymp 1$, under sub-Gaussian assumptions in the statement. 
 \end{lemma}

In order to show the remaining strict dual feasibility condition of our PDW construction, we consider \eqref{EqnThm1Q} (by the zero-subgradient and the definition of $\Q$) in the block form:
\begin{align}
	\begin{bmatrix}
    	\Q_{\Nonreg\Nonreg} & \Q_{\Nonreg\DSupNonreg} & \Q_{\Nonreg\USupNonregC}\\
    	\Q_{\DSupNonreg\Nonreg} & \Q_{\DSupNonreg\DSupNonreg} & \Q_{\DSupNonreg\USupNonregC}\\
    	\Q_{\USupNonregC\Nonreg} & \Q_{\USupNonregC\DSupNonreg} & \Q_{\USupNonregC\USupNonregC}\\
    \end{bmatrix}	
    \begin{bmatrix}
    	\Gth_{\Nonreg} - \Tth_{\Nonreg}  \\
    	\Gth_{\DSupNonreg} - \Tth_{\DSupNonreg} \\
    	{\bm 0}
    \end{bmatrix}	
    + 
    \begin{bmatrix}
    	\nabla\L(\Tth)_{\Nonreg} \\
    	\nabla\L(\Tth)_{\DSupNonreg}\\
    	\nabla\L(\Tth)_{\USupNonregC}
    \end{bmatrix}	
    + 
    \lam
    \begin{bmatrix}
    	{\bm 0} \\
    	\Gz_{\DSupNonreg} \\
    	\Gz_{\USupNonregC} \\
    \end{bmatrix}	
    =
    {\bm 0}.
\end{align}

By simple manipulation, we can obtain
\begin{align}\label{EqnCor1}
	\Gz_{\USupNonregC} = \frac{1}{\lam} \left\{ - \nabla\L(\Tth)_{\USupNonregC} + \Q_{\USupNonregC\USupNonreg} \Big(\Q_{\USupNonreg\USupNonreg}\Big)^{-1} \Big( - \nabla\L(\Tth)_{\USupNonreg} - \lam \Gz_{\DSupNonreg} \Big)  \right\}.
\end{align}
Here note that our construction of PDW can guarantee the $\ell_\infty$ bound in \eqref{EqnThmSup2}. In case of \eqref{EqnLS}, since we have $\nabla \L(\th) = \GG\th -\widehat{\gamma}$ and $\nabla^2\L(\th) = \GG$ where $(\GG,\widehat{\gamma}) = \left(\frac{\X^\top X}{n}, \frac{\X^\top \y}{n}\right)$, we need to show below that 
	\begin{align}
		\Gz_{\USupNonregC} & \ \leq \frac{1}{\lam} \left\{ - \GG_{\USupNonregC\USupNonreg}\Tth_{\USupNonreg} + \widehat{\gamma}_{\USupNonregC} + \GG_{\USupNonregC\USupNonreg}\Tth_{\USupNonreg}  - \GG_{\USupNonregC\USupNonreg} \Big(\GG_{\USupNonreg\USupNonreg}\Big)^{-1} \widehat{\gamma}_{\USupNonreg} \right\} + \matnormBig{\GG_{\USupNonregC\USupNonreg} \Big(\GG_{\USupNonreg\USupNonreg}\Big)^{-1}}{\infty} \nonumber\\
		& \ \leq \frac{1}{\lam} \left\{ \widehat{\gamma}_{\USupNonregC}  - \GG_{\USupNonregC\USupNonreg} \Big(\GG_{\USupNonreg\USupNonreg}\Big)^{-1} \widehat{\gamma}_{\USupNonreg} \right\} + \eta
	\end{align}
 	for the strict dual feasibility from \eqref{EqnCor1}.
 	As derived in \cite{LW17}, we can write 
	\begin{align}
		\left\| \widehat{\gamma}_{\USupNonregC}  - \GG_{\USupNonregC\USupNonreg} \Big(\GG_{\USupNonreg\USupNonreg}\Big)^{-1} \widehat{\gamma}_{\USupNonreg} \right\|_\infty	= \left\|\frac{\X^\top_{\USupNonregC}\Pi\e}{n}\right\|_\infty
	\end{align}
	where $\Pi$ is an orthogonal project matrix on $\X_{\USupNonreg}$: $I - \X_{\USupNonreg}(\X_{\USupNonreg}^\top \X_{\USupNonreg})^{-1}\X_{\USupNonreg}^\top$.

	For any $j$, we define $u_j$ such that $e_j^\top \frac{\X^\top_{\USupNonregC}\Pi\e}{n} := u_j^\top \e$. Then we have
	\begin{align}
		\|u_j\|_2^2 = \left\|\frac{\Pi\X_{\USupNonregC} e_j}{n} \right\|_2^2 \leq \left\|\frac{\X_{\USupNonregC} e_j}{n} \right\|_2^2 \leq \frac{\cLSsix}{n}. 	
	\end{align}
	Hence by the sub-Gaussian tail bounds followed by a union bound, we can conclude that
	\begin{align}
		\left\| \widehat{\gamma}_{\USupNonregC}  - \GG_{\USupNonregC\USupNonreg} \Big(\GG_{\USupNonreg\USupNonreg}\Big)^{-1} \widehat{\gamma}_{\USupNonreg} \right\|_\infty \leq C \sqrt{\frac{\log p}{n}} 
	\end{align}
	with probability at least $1-c \exp (-c' \log p)$ for \emph{all} selections of $\Nonreg$.
	We can establish have strict dual feasibility for any selection of $\Nonreg$ w.h.p, provided $\lam > \frac{C}{1-\eta}\sqrt{\frac{\log p}{n}}$, and now turn to $\ell_\infty$ bounds. From \eqref{EqnThmInfty}, we have
	\begin{align}
		\left\| \GG_{\USupNonreg\USupNonreg} \Big(\GG_{\USupNonreg\USupNonreg}\Tth _{\USupNonreg} -\widehat{\gamma}_{\USupNonreg} \Big) \right\|_\infty = \left\|\left(\frac{\X^\top_{\USupNonreg}\X_{\USupNonreg} }{n}\right)^{-1}\left(\frac{\X^\top_{\USupNonreg}\e}{n}\right) \right\|_\infty	.
	\end{align}
	Then for $j \in \USupNonreg$, we define $v$ such that $e_j^\top 
 \left(\frac{\X^\top_{\USupNonreg}\X_{\USupNonreg} }{n}\right)^{-1}\left(\frac{\X^\top_{\USupNonreg}\e}{n}\right)
 := v_j^\top \e$. Since for any selection of $\Nonreg$, $\|v_j\|_2^2$ is bounded as follows:
 \begin{align}
 	\|v_j\|_2^2 = \frac{1}{n^2} \left\|\X_{\USupNonreg} \left(\frac{\X^\top_{\USupNonreg}\X_{\USupNonreg} }{n}\right)^{-1} e_j \right\|_2^2 =\frac{1}{n} \left|e_j^\top \left(\frac{\X^\top_{\USupNonreg}\X_{\USupNonreg} }{n}\right)^{-1} e_j \right|_2^2 \leq \frac{\cLSsix}{n}.
 \end{align}
 Similarly by the sub-Gaussian tail bound and a union bound over $j$, we can obtain 
\begin{align}
	\left\| \GG_{\USupNonreg\USupNonreg} \Big(\GG_{\USupNonreg\USupNonreg}\Tth _{\USupNonreg} -\widehat{\gamma}_{\USupNonreg} \Big) \right\|_\infty \leq C \sqrt{\frac{\log p}{n}} 	
\end{align}
with probability at least $1-c \exp (-c' \log p)$.

\subsection{Proof of Corollary \ref{CorGlasso}}

As in the proof of Corollary \ref{CorLS}, the proof procedure is quite similar to that of Corollary 4 of \cite{LW17}. Deriving upper bounds on $\Supp$ in \cite{LW17} can be seamlessly extendable to upper bounds on $\USupNonreg$ for any selection of $\Nonreg \subseteq \{1,2,\hdots,p\} \times \{1,2,\hdots,p\} \text{ s.t. } |\Nonreg| = h$. mainly because the required upper bounds are related to entry-wise maximum on the true support $\Supp$ but entry-wise maximum in this case is uniformly upper bounded for all entries. 

Specifically, it computes the upper bound of $\|\textrm{vec}(\widehat{\Sigma}_{\Supp} - \Sigma^*_{\Supp})\|_\infty$ from the fact that $\|\textrm{vec}(\widehat{\Sigma} - \Sigma^*)\|_\infty \leq c\sqrt{\frac{\log p}{n}}$. This actually holds for any selection of $\Nonreg$. Similarly, it computes the upper bound of $\max_{(j,k)\in \Supp} |e^\top_j (\Sigma^*\Delta)^{\ell}\Sigma^*e_k|$ by H\"older's inequality and the definition of matrix induced norms: $|e^\top_j (\Sigma^*\Delta)^{\ell}\Sigma^*e_k| \leq \|e^\top_j (\Sigma^*\Delta)^{\ell-1} \|_1 \|\Delta\Sigma^*e_k\|_\infty \leq \matnormbig{(\Sigma^*\Delta)^{\ell-1}}{\infty} \|\Delta\|_{\max} \|\Sigma^*e_k \|_1 \leq \matnorm{\Sigma^*}{\infty}^{\ell+1} \matnorm{\Delta}{1}^{\ell-1} \|\Delta\|_{\max} $, which clearly holds for any index $(j,k)$ beyond $\Supp$. Finally, $\matnorm{\TQ - \nabla^2 \L(\TTh)_{\Supp\Supp}}{\infty}$ is shown to be upper bounded by the fact that $\matnorm{\TQ - \nabla^2 \L(\TTh)_{\Supp\Supp}}{\infty} \precsim d\sqrt{\frac{\log p}{n}}$.

 The remaining proof of this result directly follows similar lines to the proof of Corollary 4 in \cite{LW17}. 
 
\subsection{Proof of Theorem~\ref{th:con}}

\begin{proof}
From Algorithm~\ref{alg:bcd}, we obtain the relation
\begin{align*}
\frac{1}{\tau}(\bm w^k - \bm w^{k+1}) + r(\bm \theta^{k+1}) - r(\bm \theta^k) &\in \bm r(\bm \theta^{k+1}) + \partial \delta(\bm w^{k+1}|\cS)\\
\frac{1}{\eta}(\bm \theta^k - \bm \theta^{k+1}) + \nabla \L(\bm \theta^{k+1}) - \nabla \L(\bm \theta^k) &\in \nabla \L(\bm \theta^{k+1}) + \lambda\sum_{i=1}^p w_i^{k+1} \partial r_i(\bm \theta^{k+1})
\end{align*}
from the proximal gradient steps.

At $k$-th iteration, we have
\begin{align*}
&\L(\bm \theta^{k+1}) + \lambda\inner{\bm w^{k+1}}{\bm r(\bm \theta^{k+1})}\\
\le & \L(\bm \theta^{k}) + \inner{\nabla \L(\bm \theta^{k})}{\bm \theta^{k+1} - \bm\theta^k} + \frac{L}{2}\|\bm \theta^{k+1} - \bm \theta^k\|^2 + \lambda\inner{\bm w^{k+1}}{\bm r(\bm \theta^{k+1})}\\
\le & \L(\bm \theta^{k}) + \inner{\nabla \L(\bm \theta^k) + \lambda \sum_{i=1}^p w_i^{k+1}\partial r_i(\bm \theta^{k+1})}{\bm \theta^{k+1} - \bm\theta^k} + \frac{L}{2}\|\bm \theta^{k+1} - \bm\theta^k\|^2 + \lambda\inner{\bm w^{k+1}}{\bm r(\bm \theta^k)}\\
= &\L(\bm \theta^{k}) + \lambda\inner{\bm w^{k}}{\bm r(\bm \theta^k)} - (\frac{1}{\eta} - \frac{L}{2})\|\bm \theta^{k+1} - \bm\theta^k\|^2 + \lambda\inner{\bm w^{k+1} - \bm w^k}{\bm r(\bm \theta^k)}\\
\le & \L(\bm \theta^{k}) + \lambda\inner{\bm w^{k}}{\bm r(\bm \theta^k)} - (\frac{1}{\eta} - \frac{L}{2})\|\bm \theta^{k+1} - \bm\theta^k\|^2 + \lambda\inner{\bm w^{k+1} - \bm w^k}{\frac{1}{\tau}(\bm w^{k} - \bm w^{k+1}) - \partial\delta(\bm w^{k+1})}\\
\le & f(\bm \theta^{k}) + \lambda\inner{\bm w^{k}}{\bm r(\bm \theta^k)} - (\frac{1}{\eta} - \frac{L}{2})\|\bm \theta^{k+1} - \bm\theta^k\|^2 - \frac{\lambda}{\tau}\|\bm w^{k+1} - \bm w^k\|^2
\end{align*}
If we choose $\eta = 1/L_f$, we have,
\[
\frac{L_f}{2}\|\bm \theta^{k+1} - \bm \theta^k\|^2 + \frac{\lambda}{\tau}\|\bm w^{k+1}- \bm w^k\|^2 \le F(\bm \theta^k) - F(\bm \theta^{k+1})
\]
By telescoping both sides we get,
\[
\frac{1}{K}\sum_{k=1}^K (\frac{L_f}{2}\|\bm \theta^{k+1} - \bm \theta^k\|^2 + \frac{\lambda}{\tau}\|\bm w^{k+1}- \bm w^k\|^2) \le \frac{1}{K}(F(\bm \theta^K) - F^*).
\]
Moreover we know that,
\[
T(\bm \theta^{k+1}, \bm w^{k+1}) \le (4 + 2\lambda L_r/L_f) \mathcal{G}_k.
\]
\end{proof}

\section{Simulations for Gaussian Graphical Models.}
%\paragraph{Simulations for Gaussian Graphical Models.}

We now illustrate the usefulness trimmed regularization for sparse Gaussian Graphical Model estimation. We consider the ``diamond'' graph example described in~\cite{RWRY11} (section 3.1.1).
This graph $G = (V,E)$ has vertex set $V = \{1, 2, 3, 4\}$, with all edges except $(1, 4)$.  We consider a family of true covariance matrices  with diagonal entries $\Sigma^*_{ii} = 1$ for all $i \in V$; off-diagonal elements $\Sigma^*_{ij} = \rho$ for all edges $(i,j) \in  E \setminus \{(2, 3)\}$; $\Sigma^*_{23} = 0$; and finally the entry corresponding to the non-edge $(1, 4)$ is set as $\Sigma^*_{14 }= 2\rho^2.$  We analyze the performance of Graphical Trimmed Lasso under two settings: $\rho \in \{0,1,0.3\}.$ As discussed in~\cite{RWRY11}, if $\rho=0.1$ the incoherence condition is satisfied ; if $\rho=0.3$ it is violated. 
Under both settings, we report the probability of successful support recovery based on 100 replicate experiments for $n=100$ and $\frac{p^2-h}{p^2} \in \{0.4,0.5, ...,1\}$ and compare it with Graphical Lasso, Graphical SCAD and Graphical MCP (The MCP and SCAD parameters were set to 2.5 and 3.0 as varying these did not affect the results significantly).  For each method and replicate experiment, success is declared if the true support is recovered for at least one value of $\lambda_n$ along the solution path. We can see that for a wide range of values for the trimming parameter, Graphical Trimmed Lasso outperforms SCAD and MCP alternatives regardless of whether the incoherence condition holds or not. In addition its probability of success is always superior to that of vanilla Graphical Lasso, which fails to recover the true support when the incoherence condition is violated.  

\begin{figure}[t!]
		\centering
		%\vspace{-.2cm}
		 \includegraphics[width=0.4\linewidth]{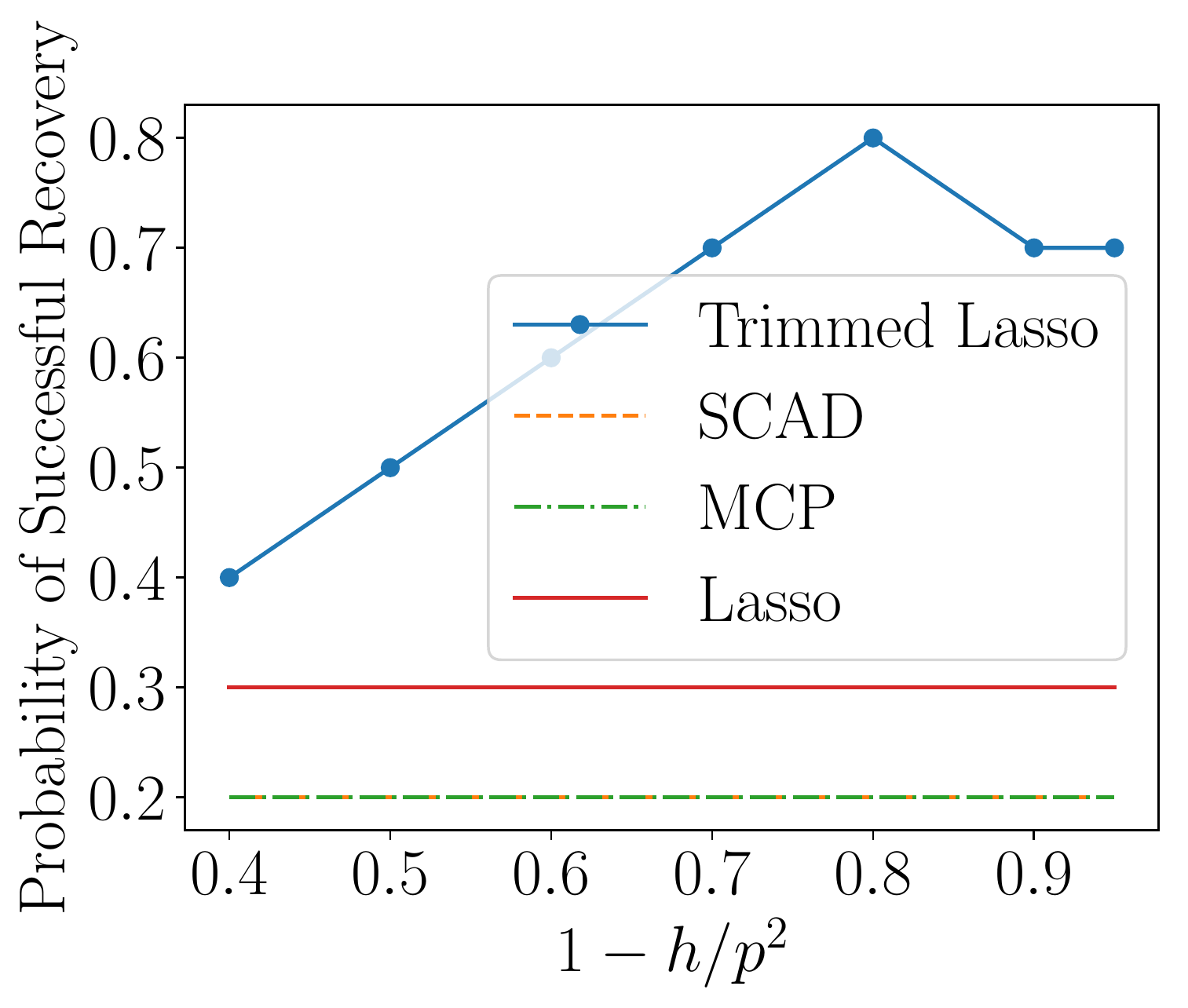} ~ \includegraphics[width=0.4\linewidth]{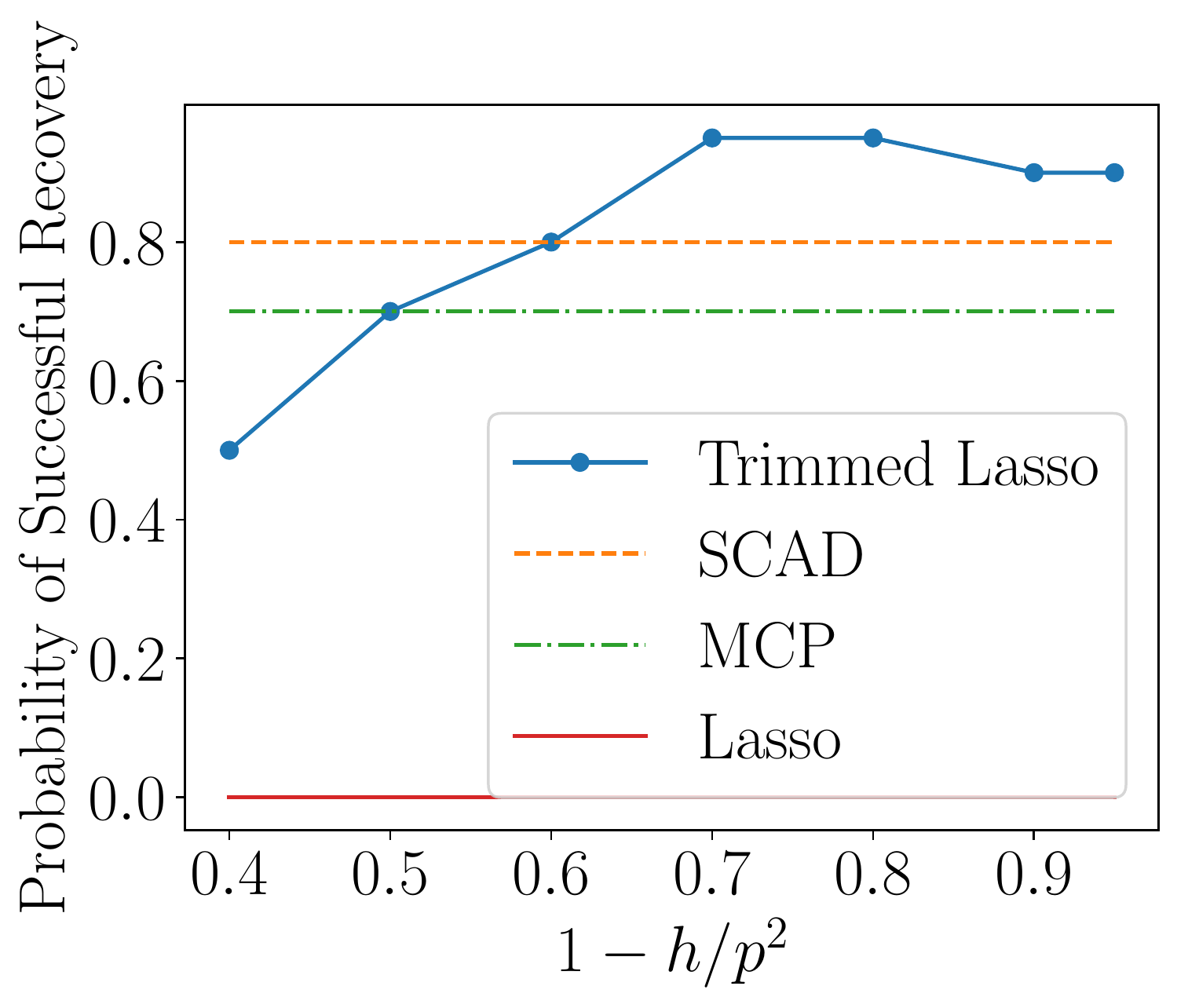} 
		\caption{Probability of successful support recovery for Graphical Trimmed Lasso as $h$ vary, Graphical SCAD, Graphical MCP and Graphical Lasso. Left: incoherence condition holds. Right: incoherence condition is violated.}
		\label{Fig:nonincoherence}
	\end{figure}

\end{document}